\documentclass[12pt]{article}  

\usepackage{amssymb}
\usepackage{amsthm}
\usepackage{graphicx}
\usepackage{amsmath}
\usepackage{amsfonts}
\usepackage{color}
\usepackage{centernot}
\usepackage{mathtools}
\usepackage{stmaryrd}

\newtheorem{proposition}{Proposition}
\newtheorem{definition}{Definition}

\newtheorem*{conjecture*}{Conjecture}
\numberwithin{equation}{section}

\newcommand{\newtext}[1]{{#1}}

\begin{document}
\title{RG approach to the inviscid limit \\ for shell models of turbulence}
\author{Alexei A. Mailybaev} 
\date{Instituto de Matem\'atica Pura e Aplicada -- IMPA, Rio de Janeiro, Brazil \\ Email: alexei@impa.br}

\maketitle

\begin{abstract}
We consider an initial value problem for shell models that mimic turbulent velocity fluctuations over a geometric sequence of scales. Our goal is to study the convergence of solutions in the inviscid (more generally, vanishing regularization) limit and explain the universality of both the limiting solutions and the convergence process. We develop a renormalization group (RG) formalism representing this limit as dynamics in a space of flow maps. For the dyadic shell model, the RG dynamics has a fixed-point attractor, which determines universal limiting solutions. Deviations from the limiting solutions are also universal and given by a leading eigenmode (eigenvalue and eigenvector) of the linearized RG operator. \newtext{Application to the Gledzer shell model reveals the RG attractor in the form of a closed invariant curve, while the Sabra shell model yields chaotic RG dynamics. An important consequence of the RG formalism is the understanding of the different roles of symmetry-preserving (canonical) and symmetry-breaking (e.g. viscous) regularizations.}
\end{abstract}

\section{Introduction} 

Physical models of ideal fluid and wave dynamics can be ill-posed, e.g., as a consequence of blowup in a finite time~\cite{eggers2015singularities}. Classical examples include, among many others, the inviscid Burgers equation~\cite{dafermos2005hyperbolic} and the Euler equations for incompressible ideal fluid~\cite{luo2014potentially}. A common way to define solutions at all times is to consider a regularized system, e.g. by adding viscous terms. Since these terms are small, the question arises whether there is a limit of vanishing regularization. As a paradigmatic example, this approach provides shock wave solutions to the inviscid Burgers equation. Numerical analysis shows that these limiting solutions exist and are universal for a large class of regularizations, for example, when replacing viscous terms with hyperviscous ones. It has been suggested that a similar scenario for the Euler system holds in a stochastic formulation in which the regularization includes both viscous forces and small-scale noise~\cite{thalabard2020butterfly}. Such solutions are called spontaneously stochastic, because they remain probabilistic even after the noise is removed in the limit of vanishing regularization. 

In this paper we address the general question: why the limit of vanishing regularization converges and to what extent the limiting process is universal with respect to the choice of regularization? We consider shell models of turbulence, which simulate physical space using a geometric sequence of scales and allow very precise numerical investigation~\cite{biferale2003shell}. 
Specifically, we consider the dyadic (Desnyansky--Novikov) model~\cite{desnyansky1974evolution} that mimics the Burgers equation~\cite{cheskidov2009vanishing,mailybaev2015continuous}, as well as the Gledzer and Sabra models~\cite{gledzer1973system,l1998improved} related to the Navier--Stokes turbulence~\cite{frisch1999turbulence}. We limit our consideration to deterministic regularizations, leaving a similar study of stochastic regularizations for a forthcoming paper.

Our main result is the development of the renormalization group (RG) formalism. It describes the limit of vanishing regularization as the dynamics of an RG operator acting in the space of flow maps. A similar formalism was presented earlier for models on a discrete space-time lattice~\cite{mailybaev2023spontaneously,mailybaev2023spontaneous,mailybaev_RG_2025}. The extension to shell models in which time is continuous is carried out by introducing the concept of canonical (symmetry preserving) regularizations. For the dyadic model, the RG approach not only justifies the existence and universality of the limiting solutions as a fixed-point attractor of the RG dynamics, but also predicts new universal properties. This new universality is related to the leading eigenmode (eigenvalue and eigenvector) of the linearized RG operator, which determine deviations of regularized solutions from their limiting form. 
The usual viscous regularization is not canonical (it violates time scale invariance), but its RG analysis is mediated by auxiliary canonical regularizations.
We also apply the RG approach to the Gledzer and Sabra shell models. In the former case, \newtext{we show that the attractor of the RG dynamics is an invariant closed curve, which  governes the vanishing regularization limit.}
For the Sabra model, the RG dynamics is chaotic with double exponential divergence of solutions. 

In general, an RG formalism explores self-similarity of a system manifested at a large range of scales, but its precise form may vary a lot depending on the problem. In our case, this is the self-similarity of equations of motion corresponding to the ideal system. As we showed in~\cite{mailybaev2023spontaneous}, our RG approach has much in common with the Feigenbaum theory~\cite{feigenbaum1983universal}; in particular, one can establish the analytical similarity for the model of digital turbulence. 
Opposite to the Kadanoff--Wilson RG approach~\cite{wilson1983renormalization} and its extension to fluid dynamics \cite{forster1977large,yakhot1986renormalization,canet2022functional}, we do not coarse-grain system properties but rather keep all details at small scales intact. One can also recognize in our formalism some features of the inverse RG~\cite{eyink1994analogies,inverseRG,kupiainen2007scaling}, because our RG operator acts by adding an extra largest scale and reconstructing the ideal dynamics at that scale. 

The paper is organized as follows. Section~\ref{sec2} introduces shell models.  Section~\ref{sec_canon} defines the RG operator and relates it to a class of canonical regularizations. Section~\ref{sec_fixedpoint} studies the limit of vanishing regularization in terms of the fixed-point RG attractor for the dyadic model. Section~\ref{subsec_visc} extends the results to the (non-canonical) viscous regularization. Section~\ref{sec_hopf} studies the RG attractor in the Gledzer shell model. Section~\ref{sec_sabra} describes the chaotic RG dynamics in the Sabra model. Section~\ref{sec_disc} summarizes the results and discusses further developments. Some technical derivations are gathered in the Appendix.

\section{Ideal and viscous dyadic models} \label{sec2}

Shell models of turbulence mimic ideal and viscous fluid dynamics using a geometric sequence of scales $ l_n = \lambda^{-n}$, where $\lambda > 1$ is the inter-shell ratio and $n$ are integer shell numbers. In this paper, we set $\lambda = 2$. The associated wavenumbers are defined as $k_n = 1/ l_n = \lambda^n$. Each scale is represented by a shell variable $u_n$. In this and subsequent sections, we consider the dyadic (Desnyansky--Novikov) shell model~\cite{desnyansky1974evolution} with real variables $u_n \in \mathbb{R}$. In Sections~\ref{sec_hopf} and \ref{sec_sabra}, we will extend our approach to the Gledzer and Sabra shell model~\cite{l1998improved} with real and complex variables. 

Equations of the ideal shell model (a toy model for the inviscid Burgers equation or Euler equations of ideal fluid) are formulated as
	\begin{equation}
    	\frac{du_n}{dt} = k_n f_n, \quad n \ge 1, 
	\label{eq1_1a}
    	\end{equation}
for positive shell numbers $n$ and time $t \ge 0$. In the dyadic model, the coupling term $f_n$ takes the form
     	\begin{equation}
	f_n := f(u_{n-1},u_n,u_{n+1}) = u_{n-1}^2- \lambda u_n u_{n+1}. 
	\label{eq1_1fn}
    	\end{equation}
Note that although we use the specific form (\ref{eq1_1fn}), our theoretical construction extends to general homogeneous functions $f$.

We study the initial-boundary value problem (IBVP). The ideal IBVP is given by the system (\ref{eq1_1a}) with the initial conditions 
    	\begin{equation}
	u_n(0) = a_n \ \ \textrm{for} \ \  n \ge 1,
	\label{eq1_IC}
    	\end{equation}
and the boundary condition
    	\begin{equation}
	u_0(t) = b(t) \ \ \textrm{for} \ \ t \ge 0, 
	\label{eq1_BC}
    	\end{equation}
where $b \in C^1$ is a given continuously differentiable function. We will use the short notations $u = (u_n)_{n \ge 1} = (u_1,u_2,\ldots)$ and $a = (a_n)_{n \ge 1}$ for the respective infinite sequences. Physical considerations require that the energy $E = \sum u_n^2 < \infty$ is finite at all times, i.e., the sequences $u(t)$ and $a$ are square-summable. 

The ideal IBVP has two scaling symmetries. The space scaling is formulated as
     	\begin{equation}
	\tilde{u}_n(t) = \lambda u_{n+1}(t), 
	\quad \tilde{a}_n = \lambda a_{n+1}, 
	\quad \tilde{b}(t) = \lambda u_1(t),
	\label{eq1_2SSa}
    	\end{equation}
where the tildes denote a new solution for new initial and boundary conditions. 
\newtext{It is useful to write Eq.~(\ref{eq1_2SSa}) as
     	\begin{equation}
	\tilde{u}(t) = S^+u(t), 
	\quad \tilde{a} = S^+a, 
	\quad \tilde{b}(t) = \lambda u_1(t),
	\label{eq1_2SS}
    	\end{equation}
where we define the scaling operators $S^+a = (\lambda a_2,\lambda a_3,\lambda a_4,\ldots)$ and $S^-a = (0,a_1/\lambda,a_2/\lambda,\ldots)$.} Similarly, the time scaling symmetry takes the form
     	\begin{equation}
	\tilde{u}(t) = \alpha u(\alpha t), 
	\quad \tilde{a} = \alpha a, 
	\quad \tilde{b}(t) = \alpha b(\alpha t),
	\label{eq1_2TS}
    	\end{equation}
where $\alpha > 0$ is an arbitrary positive factor.

The ideal IBVP is generally ill-posed~\cite{constantin2007regularity,cheskidov2023dyadic}. The natural (physically motivated) way is to regularize the system by adding a viscous term to Eqs.~(\ref{eq1_1a}). The resulting equations read
    	\begin{equation}
    	\frac{du_n}{dt} = k_nf_n - \nu k_n^{2} u_n, \quad n \ge 1,
	\label{eq1_V1}
    	\end{equation}
where $\nu > 0$ is the viscosity parameter. 
The viscous IBVP consists of the system~(\ref{eq1_V1}) with initial conditions (\ref{eq1_IC}) and boundary condition (\ref{eq1_BC}). 
Under proper assumptions~\cite{constantin2006analytic,filonov2017uniqueness}, the viscous IBVP is well-posed, possessing unique solutions at all positive times. The ideal model is recovered in the inviscid limit, $\nu \to 0$.
Thus, physically relevant solutions of the ideal IBVP can be sought in the limit of vanishing viscosity.

\section{Canonical regularizations and RG operator} \label{sec_canon}

In this section, we consider regularizations from a general point of view. A regularization, for which the viscous model is an example, is supposed to change the ideal model  such that the respective IBVP is well-posed, i.e., has unique global-in-time solutions. Hence, we can identify a particular regularization with a family $\Phi = \{\Phi_t\}_{t \ge 0}$ of maps $\Phi_t$ providing regularized solutions $u(t)$ as
	\begin{equation}
	u(t) = \Phi_t(a,b).
	\label{eq2_FMa}
    	\end{equation}
\newtext{Therefore, the maps $\Phi_t$ are defined as functions of two arguments: the initial condition $a = (a_n)_{n \ge 1}$ and the boundary function $b \in C^1$. 
In analogy with the dynamical systems theory, we call $\Phi$ a flow map of a regularized IBVP.  
The flow map is assumed to have the following properties. Every solution (\ref{eq2_FMa}) must satisfy the respective initial conditions (\ref{eq1_IC}). The causality property requires that the value of $\Phi_t(a,b)$ at a given time depends on the boundary function $b$ only from the past time interval $[0,t]$.
Finally, $\Phi_t$ satisfies the condition analogous to the usual semigroup relation, which we formulate as
	\begin{equation}
	\Phi_{t+s}(a,b) = \Phi_{s}(a',b'), \quad t,s \ge 0,
	\label{eq2_FMb}
    	\end{equation}
where $a' = \Phi_{t}(a,b)$ with the the time-shifted boundary condition $b'(t') = b(t+t')$. This time shift takes into account non-autonomous dynamics due to time-dependent boundary condition.}

We say that $\Phi$ is the $N$-level regularization if every solution (\ref{eq2_FMa}) satisfies the ideal model Eq.~(\ref{eq1_1a}) for the shells $n = 1,\ldots,N$. The idea is to transfer this regularization to smaller scales using the scaling symmetry (\ref{eq1_2SS}), thereby, constructing the $(N+1)$-level regularization. 
This procedure is described by the following

\newtext{
\begin{definition}
\label{def1}
A flow map $\Phi$ is called renormalizable if, for any $a$ and $b$, there exists a unique solution $u_1(t)$ of the initial value problem
	\begin{equation}
    	\frac{du_1}{dt} = k_1 f(b,u_1,u_2), \quad u_1(0) = a_1,
	\label{eq2_FMe}
    	\end{equation}
where $u_2(t) = \tilde{u}_1(t)/\lambda$ with 
	\begin{equation}
	\tilde u(t) = \Phi_t(\tilde a,\tilde b), \quad
 	\tilde a = S^+a; \quad
	\tilde b(t) = \lambda u_1(t).
 	\label{eq2_FMf}
    	\end{equation}
Then, we introduce the RG operator $\Phi \mapsto \mathcal{R}[\Phi]$ acting on renormalizable flow maps as 
	\begin{equation}
 	\label{eq2_FMff}
	\mathcal{R}[\Phi]_t(a,b) = S^-\tilde u(t)+(u_1(t),0,0,\ldots) = (u_1(t),\tilde{u}_1(t)/\lambda,\tilde{u}_2(t)/\lambda,\ldots).
    	\end{equation}
\end{definition}
}

\begin{proposition}
\label{prop1}
If $\Phi$ is a renormalizable $N$-level regularization, then $\mathcal{R}[\Phi]$ is an $(N+1)$-level regularization.
\end{proposition}

\begin{proof}
\newtext{From the definition of the scaling operators it is clear that $S^-S^+ a = (0,a_2,a_3,\ldots)$, which differs from the identity in the first component.}
Hence, one can see that the relations (\ref{eq2_FMf}) and (\ref{eq2_FMff}) coincide with the scaling relation (\ref{eq1_2SS}) for the components $n \ge2$ of the solution $u(t) = \mathcal{R}[\Phi]_t(a,b)$. 
Since $\tilde u(t)$ is generated by the $N$-level regularization $\Phi$, it satisfies 
 the ideal Eq.~(\ref{eq1_1a}) for the shells $n = 1,\ldots,N$.
Then one can check that $u_n(t)$ satisfy the ideal Eq.~(\ref{eq1_1a}) for the shells $n = 2,\ldots,N+1$. The first relation in Eq.~(\ref{eq2_FMe}) provides the remaining ideal equation for $n = 1$. 
\end{proof}

For now, Definition~\ref{def1} has an important flaw. \newtext{Namely, combining the two symmetries (\ref{eq1_2SS}) and (\ref{eq1_2TS}), one can replace $\tilde{u}(t)$ in Eq.~(\ref{eq2_FMf}) by
	\begin{equation}
	\tilde u(t) = \frac{1}{\alpha}\Phi_{t/\alpha}(\tilde a,\tilde b), \quad
 	\tilde a = \alpha S^+a; \quad
	\tilde b(t) = \alpha \lambda u_1(\alpha t).
	\label{eq2_FMeT}
    	\end{equation}
} It is straightforward to check that Proposition~\ref{prop1} remains valid for this new formulation. This means that our definition of the RG operator is ambiguous. This ambiguity is eliminated by considering a specific subclass of canonical regularizations, as we show next.

We say that the regularization is time-scale invariant if solutions (\ref{eq2_FMa}) obey the symmetry relations (\ref{eq1_2TS}). 
In terms of the flow map, this condition is formulated as \newtext{
	\begin{equation}
	\Phi_{t}(a,b) = \frac{1}{\alpha} \Phi_{t/\alpha}(\tilde a,\tilde b), \quad 
	\tilde a_n = \alpha a_n, \quad
	\tilde b(t) = \alpha b(\alpha t),
	\label{eq2_FMd}
    	\end{equation}
for any $a$, $b$ and $\alpha > 0$. Clearly, this condition guarantees that $\tilde u(t)$ in Eq.~(\ref{eq2_FMeT}) does not depend on $\alpha$.}

\begin{definition}
\label{def2}
We say that a flow map $\Phi$ is a canonical regularization if both $\Phi$ and its RG iterations $\mathcal{R}^N[\Phi]$ for $N \ge 1$ are time-scale invariant and renormalizable.
\end{definition}

We remark that the time-scale invariance of $\Phi$ already implies the time-scale invariance of $\mathcal{R}^N[\Phi]$ for all $N \ge 1$. For canonical regularizations, expressions (\ref{eq2_FMe}) and (\ref{eq2_FMeT}) are equivalent because the time scaling does not change the flow map. Hence, we define the RG operator unambiguously by restricting its action to the space of canonical regularizations. 
By Proposition~\ref{prop1}, regularized solutions $u(t) = \mathcal{R}^N[\Phi]_t(a,b)$ satisfy the ideal model equations at shells $n = 1,\ldots,N$. Thus, all ideal equations are recovered in the limit $N \to \infty$. This property defines the vanishing regularization limit in the space of canonical regularizations as the RG dynamics $\mathcal{R}^N[\Phi]$ with $N \to \infty$. 

The important property of the RG operator $\mathcal{R}$ is that it is uniquely defined by the ideal model. Indeed, the ideal coupling function $f$ is the only model-dependent element in Definition~\ref{def1}. 
The space of canonical regularizations (the domain of $\mathcal{R}$) also depends on properties (symmetries) of the ideal model only. 
On the contrary, all the information on how solutions are regularized is contained in the flow map $\Phi$ itself. These features highlight the distinct roles of the ideal system and regularization in the RG dynamics: the first determines the RG operator $\mathcal{R}$, and the second the initial flow map $\Phi$.

Let us give concrete examples of canonical regularizations.
Given a positive integer $J$, consider the regularized IBVP governed by the system of $N+J$ differential equations
	\begin{equation}
	\frac{du_n}{dt} = k_n \left\{\begin{array}{ll}
	f_n, & n = 1,\ldots,N; \\[2pt]
	f_n-|u_n| u_n, & n = N+1,\ldots,N+J;
	\end{array}\right.
	\label{eq2_E1}
    	\end{equation}
and vanishing shell variables $u_n(t) = 0$ for $n > N+J$ and $t > 0$. This system is time-scale invariant, since the extra dissipative terms are quadratic. The  cutoff at shell $N+J$ guarantees that the respective IBVP is well-posed; see Appendix~\ref{subsec_A1}. Given the values of $J \ge 1$ and $N \ge 0$, we denote the flow map of the respective IBVP as $\Phi^{(N,J)}$. It is straightforward to check (see Appendix~\ref{subsec_A2}) that $\Phi^{(N,J)}$ are canonical regularizations and the RG operator given by Definition~\ref{def1} acts as
	\begin{equation}
	\Phi^{(N+1,J)} = \mathcal{R}[\Phi^{(N,J)}].
	\label{eq2_E2}
    	\end{equation}
Thus, for each fixed $J$, the vanishing regularization limit $N \to \infty$ is represented by the RG dynamics starting from $\Phi^{(0,J)}$.

The example (\ref{eq2_E1}) provides a recipe for an explicit construction of a large class of canonical regularizations, along with the sequence generated by the RG operator. One can simply replace $|u_n|u_n$ in Eq.~(\ref{eq2_E1}) by other types of quadratic dissipative terms. The sharp cutoff at shell $N+J$ used in Eq.~(\ref{eq2_E1}) is a convenient but not a necessary property, although proving the well-posedness of the IBVP without a cutoff is a difficult task in general. 
We consider one example without cutoff in Section~\ref{subsec_visc}.

Finally, we observe that the viscous regularization (\ref{eq1_V1}) is not canonical, because the dissipative term is linear (not quadratic). This leads to important consequences that we investigate later in Section~\ref{subsec_visc}.
The cutoff of models (\ref{eq2_E1}) mimics Large Eddy Simulation (LES) closures in fluid dynamics~\cite{pope2000turbulent}: the dissipative term can be written as $k_n|u_n|u_n = \nu_n(u) k_n^2u_n$ with the effective eddy-viscosity $\nu_n(u) = |u_n|/k_n$. 

\section{Fixed-point RG attractor} 
\label{sec_fixedpoint}

Let us consider a sequence of regularizations $\Phi^{(N)}$ with $N = 0,1,2,\ldots$ generated by the RG dynamics 
	\begin{equation}
	\Phi^{(N+1)} = \mathcal{R}[\Phi^{(N)}], 
	\label{eqSRG_1D}
	\end{equation}
where the initial flow map $\Phi^{(0)}$ is a given canonical regularization. 
Exploiting analogy with the dynamical systems theory, we now formulate two natural conjectures about the RG dynamics and then verify them numerically. The first conjecture is the existence of the fixed point RG attractor $\Phi^{\infty}$ with the property
	\begin{equation}
	\Phi^{(N)} \to \Phi^{\infty} 
	\quad \textrm{as} \quad N \to \infty, \quad 
	\Phi^{(0)} \in \mathcal{B}(\Phi^{\infty}),
	\label{eqSRG_R2}
	\end{equation}
where $\mathcal{B}(\Phi^{\infty})$ is the basin of attraction in the space of canonical regularizations. 
We must specify in which sense the limit in Eq.~(\ref{eqSRG_R2}) is understood.
Let $u^{(N)}(t) = \Phi^{(N)}_t(a,b)$ and $u^{\infty}(t) = \Phi^\infty_t(a,b)$ be the regularized and limiting solutions for any given initial and boundary conditions. Then, the limit (\ref{eqSRG_R2}) signifies that the sequence $u_n^{(N)}(t) \xrightarrow{N \to \infty} u_n^\infty(t)$ converges for any $n$ uniformly in finite time intervals. 
We stress that we do not prove the convergence in this paper, but rather provide a convincing numerical evidence of convergence motivating the above definition. 

An immediate consequence of the limit (\ref{eqSRG_R2}) is that the limiting flow map provides solutions $u^\infty(t) = \Phi^{\infty}_t(a,b)$ of the ideal IBVP.
Moreover, these limiting solutions are universal: they do not depend on the choice of regularization, as long as the initial flow map $\Phi^{(0)}$ belongs to the basin of the attraction. Taking the limit $N \to \infty$ in both sides of Eq.~(\ref{eqSRG_1D}), one can see that $\Phi^\infty$ satisfies the fixed-point condition $\Phi^{\infty} = \mathcal{R}[\Phi^{\infty}]$. This fixed-point condition is understood in the sense that Eqs.~(\ref{eq2_FMe})--(\ref{eq2_FMff}) of Definition~\ref{def1} are satisfied for $\Phi = \Phi^\infty$ and $\mathcal{R}[\Phi] = \Phi^\infty$. 

The second conjecture refers to the linearized RG dynamics. We assume that, for canonical regularizations $\Phi$ sufficiently close to the fixed-point $\Phi^\infty$, the RG operator has the linear approximation
	\begin{equation}
	\mathcal{R}[\Phi] \approx \Phi^\infty + \delta\mathcal{R}^\infty [\Psi], \quad
	\Psi = \Phi-\Phi^{\infty},
	\label{eqSRG_LA}
	\end{equation}
where $\delta\mathcal{R}^\infty$ is a variational derivative of the RG operator. Using this relation, one defines the linearized RG dynamics as
	\begin{equation}
	\Psi^{(N+1)} = \delta\mathcal{R}^\infty [\Psi^{(N)}]
	\label{eqSRG_LD}
	\end{equation}
for the deviations $\Psi^{(N)} = \Phi^{(N)}-\Phi^{\infty}$. We conjecture that the linearized RG dynamics (\ref{eqSRG_LD}) in the limit $N \to \infty$ is governed by the eigenmode solution
	\begin{equation}
	\Psi^{(N)} \approx c \rho^N \Omega,
	\label{eqSRG_1X}
	\end{equation}
where $\rho$ is a leading (largest absolute value) real eigenvalue and $\Omega = \{\Omega_t\}_{t \ge 0}$ a corresponding eigenvector of the eigenvalue problem
	\begin{equation}
	\delta\mathcal{R}^\infty [\Omega] = \rho \,\Omega.
	\label{eqSRG_1Y}
	\end{equation}
Both $\rho$ and $\Omega$ are universal in the sense that they are determined by the RG operator and its fixed point, and not by a specific regularization. 

For the flow maps, the asymptotic expression (\ref{eqSRG_1X}) yields
	\begin{equation}
	\Phi^{(N)} \approx \Phi^\infty+c\rho^N \Omega
	\quad \textrm{as} \quad N \to \infty.
	\label{eqRG_R3}
	\end{equation}
This relation is the new universality property: not only the limiting solution $u^\infty(t) = \Phi_t^\infty(a,b)$, but also the deviations 
	\begin{equation}
	\delta u^{(N)}(t) := 
	u^{(N)}(t)-u^\infty(t) \approx c\rho^N \Omega_t(a,b) 
	\label{eqRG_R3D}
	\end{equation}
of regularized solutions are universal up to a constant factor $c$. The factor $c$ is the only quantity in the right-hand sides of Eqs.~(\ref{eqRG_R3}) and (\ref{eqRG_R3D}) depending on the specific regularization sequence $\{\Phi^{(N)}\}_{N \ge 0}$.

We now verify both relations (\ref{eqSRG_R2}) and (\ref{eqRG_R3})  numerically for the regularized systems~(\ref{eq2_E1}). 
In numerical simulations, we consider three different canonical regularization with the parameters $J = 1,2,3$. The respective RG iterations are given by Eqs.~(\ref{eq2_E2}) and (\ref{eqRG_R3}) as $\Phi^{(N,J)} \approx \Phi^{\infty}+c_J\rho^N \Omega$. In this expression, the dependence on the regularization model $(N,J)$ reduces to the single constant factor $c_J\rho^N$. 
For numerical simulations, we consider two different initial conditions
	\begin{equation}
	\mathrm{IC}_1: \ \ a_n = 2^{-k_n}, \quad 
	\mathrm{IC}_2: \ \ a_n = k_n^{-1/4}(2-\sin n),
	\label{eqS3_2y}
	\end{equation}
which represent the regular (decaying exponentially in $k_n$) and rough  (power-law in $k_n$) states. 
For the boundary function we take 
	\begin{equation}
	b(t) = 2-\cos t. 
	\label{eqS3_2bc}
	\end{equation}
\newtext{Numerical simulations of the regularized systems~(\ref{eq2_E1}) are performed using the \texttt{ode45} and \texttt{ode15s} solvers in Matlab with very high accuracy. 
Note that we do not need to solve the RG Eqs.~(\ref{eq2_FMe})--(\ref{eq2_FMff}) numerically. In fact, we already proved that they are satisfied by solutions of regularized systems (\ref{eq2_E1}). Therefore, the RG formalism is only used for interpreting the observed behavior in the inviscid limit.}

\begin{figure}[tp]
\centering
\includegraphics[width=0.95\textwidth]{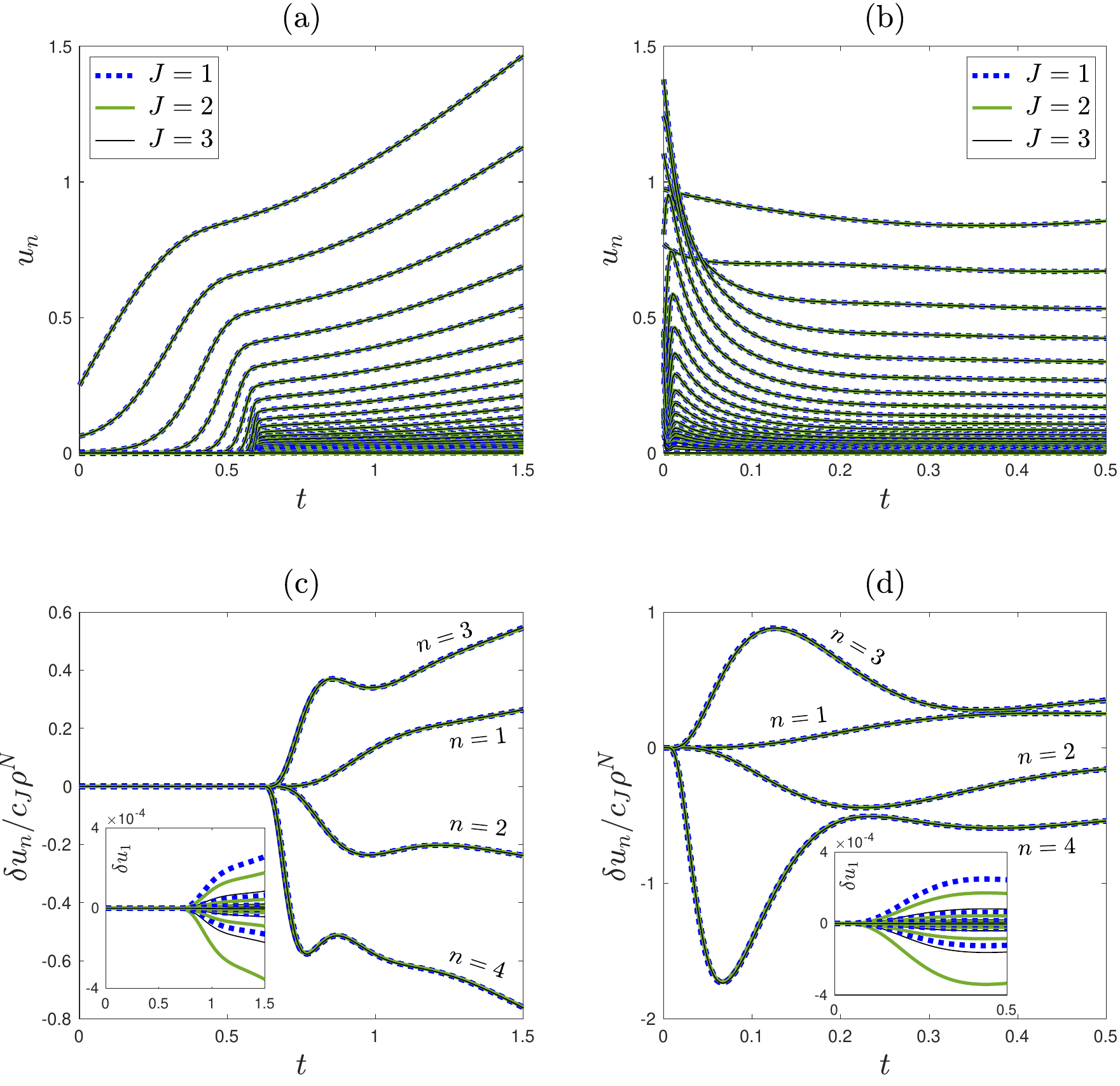}
\caption{Evolution of the regularized models with $J = 1,2,3$ and $N = 20$ from (a) regular initial  condition $\mathrm{IC}_1$ and (b) rough initial condition $\mathrm{IC}_2$. Collapse of the graphs for different models confirms the universality of the limit $N \to \infty$. Panels (c) and (d) show the corresponding deviations $\delta u_n(t)$ from the limiting solutions for $J = 1,2,3$ and $N = 10,\ldots,20$. The main plots present the rescaled graphs $\delta u_n/c_J\rho^N$, and their independence of regularization confirms the RG eigenmode asymptotic (\ref{eqRG_R3}). For comparison, the insets in panels (c and d) show the original deviations $\delta u_1(t)$.}
\label{fig1}
\end{figure}

First, let us consider the regular initial condition $\mathrm{IC}_1$ in Eq.~(\ref{eqS3_2y}). Numerical results for the three models with $N = 20$ and $J = 1,2,3$ are shown in Fig.~\ref{fig1}(a). This figure presents solutions $u_n(t)$ for different shell numbers $n = 1,2,\ldots$ (lower curves correspond to larger $n$). The graphs for different values of $J$ are visually indistinguishable, confirming the convergence of solutions for large $N$ independently of the regularization. Notice that the limiting (ideal model) solution blows up~\cite{dombre1998intermittency,mailybaev2012renormalization} at time $t \approx 0.61$. The inviscid limit extends this solution after the blowup. Figure~\ref{fig1}(b) presents analogous results for the rough initial condition $\mathrm{IC}_2$ in Eq.~(\ref{eqS3_2y}). These solutions also converge to a solution of the ideal IBVP independently of the regularization.

Figures~\ref{fig1}(c,d) verify the predictions (\ref{eqRG_R3}) for both initial conditions. The insets show the deviations of the first component $\delta u_1(t) = u_1(t)-u_1^{\infty}(t)$ for the regularized models with $J = 1,2,3$ and $N = 10,\ldots,20$. Here $u(t) = \Phi^{(N,J)}_t(a,b)$ are obtained by solving the IBVP for the regularized model~(\ref{eq2_E1}), and the limiting solution $u^{\infty}(t)$ is approximated by taking $N = 40$. The main panels (c and d) present the rescaled deviations $\delta u_n(t)/c_J\rho^N$ for the first four shells $n = 1,\ldots,4$. Here we are allowed to set $c_1 = 1$, because we did not normalize the eigenvector, and then estimate $c_2 = -1.38$ and $c_3 = 0.66$ for the remaining two types of regularization. The eigenvalue turns out to be $\rho = -1/2$, and this value is justified analytically in the Appendix~\ref{subsec_A3}. The accurate collapse of the rescaled deviations for three different regularized models $J = 1,2,3$ and the wide range of RG iterations $N = 10,\ldots,20$ is the convincing numerical verification of the asymptotic relation~(\ref{eqRG_R3}). We remark that the RG eigenmode vanishes at pre-blowup times in Fig.~\ref{fig1}(c).

\section{RG approach to the viscous model}
\label{subsec_visc}

The concept of canonical regularization is determined solely by properties of the ideal system, in particular, by the time-scale invariance of the dyadic model. Therefore, it is natural to expect that a regularization originating from different physical mechanisms is not necessarily canonical. Indeed, this is precisely the case of the viscous model (\ref{eq1_V1}), which mimics the physical dissipative mechanism. Since the viscous term is not time-scale invariant, the viscous model does not belong to the class of canonical regularizations. 

Numerical simulations indicate that the limit $\nu \to 0$ in the viscous model yields the same solution as the limit $N \to \infty$ in canonical regularizations; see Fig.~\ref{fig2}(a). On the contrary, the deviations $\delta u_n(t)$ from the limiting solution cease to be universal,  as demonstrated by the green curve in Fig.~\ref{fig2}(b).  We now show how these properties as well as certain other universal or non-universal features of solutions can be explained using the RG formalism. 

Our analysis will systematically use an empirical observation that the limiting solutions have the (Kolmogorov) asymptotic form 
	\begin{equation}
	u_n^\infty(t) \approx \gamma_t(a,b)k_n^{-1/3} \quad
	\textrm{as} \quad n \to \infty,
	\label{eqVM_E3u}
	\end{equation}
where the real prefactor $\gamma_t(a,b)$ depends on time, initial and boundary conditions but does not dependent on $n$. There is no proof of this relationship, but some related rigorous results are known~\cite{cheskidov2009vanishing}. The analysis presented below combines Eq.~(\ref{eqVM_E3u}) with our RG approach, leading to non-trivial predictions for the viscous model. 

\begin{figure}[tp]
\centering
\includegraphics[width=0.9\textwidth]{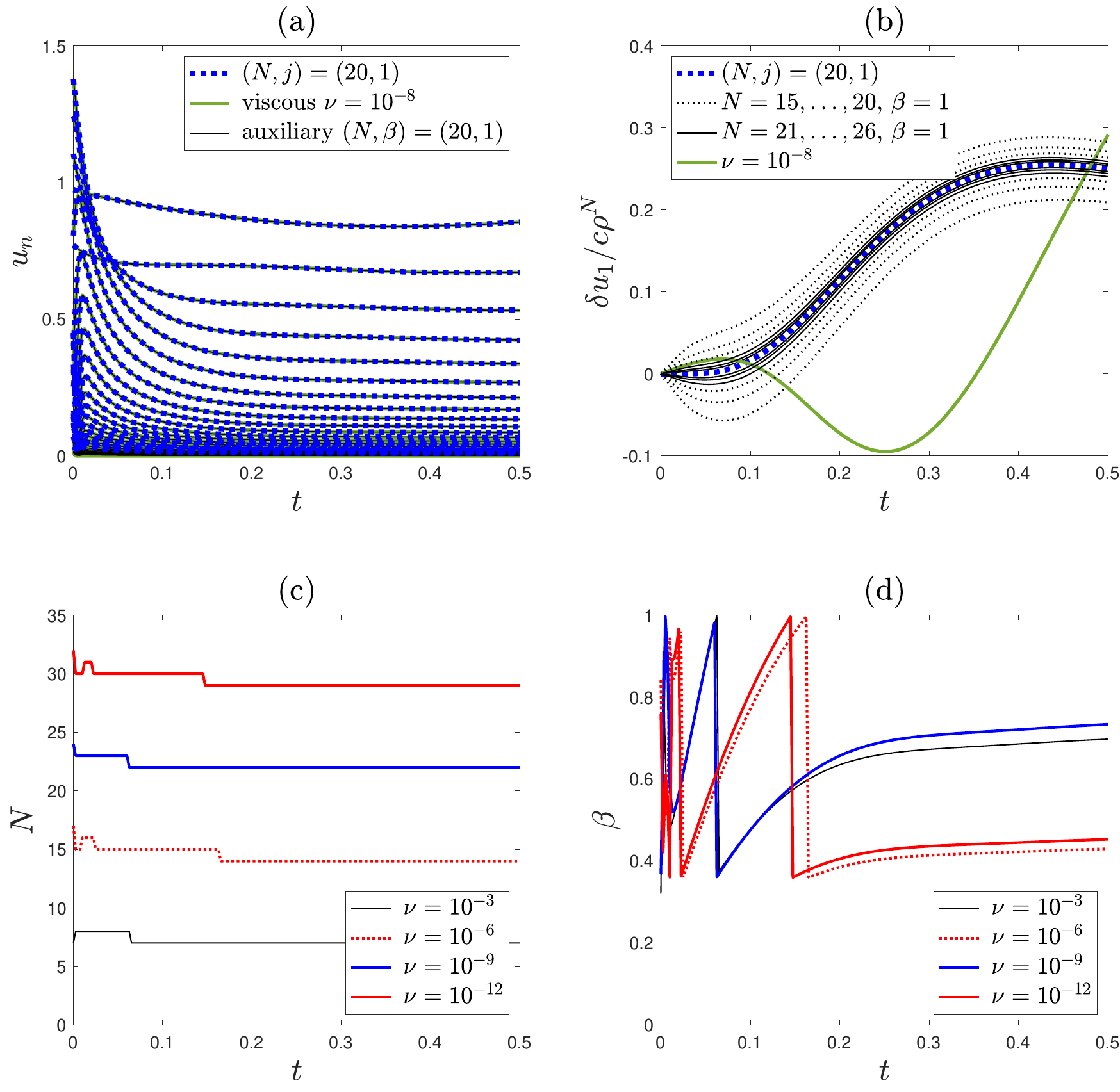}
\caption{(a) Regularized IBVP solutions for different models with the initial condition $\mathrm{IC}_2$. Their collapse confirms the convergence to the universal solution of the ideal IBVP. (b) Normalized deviations $\delta u_1(t)$ for different regularization models. Collapse of the graphs for large $N$ confirms the universal eigenmode correction for the canonical models $(N,J)$ and $(N,\beta)$. For the viscous model, the normalized deviation (green line) has a different (non-universal) shape. (c,d) Dependence of $N_t^{(\nu)}(a,b)$ and $\beta_t^{(\nu)}(a,b)$ on time for decreasing viscosities $\nu$.}
\label{fig2}
\end{figure}

\subsection{Auxiliary regularized model} 
\label{subsec_aux}

Let us first introduce and analyze an auxiliary system of the form
    	\begin{equation}
	\frac{du_n}{dt} = k_nf_n
	-\beta \frac{|u_N|}{k_N} k_n^{2} u_n.
	\label{eqVM_E1}
    	\end{equation}
In this system, we introduced an integer parameter $N \ge 0$ and a real parameter $\beta > 0$. We assume that the corresponding IBVP is well-posed: we do not have a proof but numerical simulations suggest that solutions do not blow up. We denote the corresponding flow maps by $\Phi^{(N,\beta)}$. 
One can see that system (\ref{eqVM_E1}) is time-scale invariant.

Our next step is to show that the flow maps $\Phi^{(N,\beta)}$ are related by the RG operator. For this purpose, let us analyze Eq.~(\ref{eqVM_E1}) for $n = 1$ having the form
    	\begin{equation}
	\frac{du_1}{dt} = k_1 f_1
	-\beta \frac{|u_N|}{k_N} k_1^{2} u_1.
	\label{eqVM_E2}
    	\end{equation}
Using Eq.~(\ref{eqVM_E3u}), the last term in Eq.~(\ref{eqVM_E2}) is estimated as
	\begin{equation}
	\beta \frac{|u_N|}{k_N} k_1^{2} u_1 \approx k_N^{-4/3} \beta \gamma_t(a,b) k_1^{2} u_1(t)
	\to 0 \quad \textrm{as} \quad N \to \infty.
	\label{eqVM_E3}
	\end{equation}
Hence, the first-shell Eq.~(\ref{eqVM_E2}) takes the ideal form asymptotically for large $N$, i.e., it yields Eq.~(\ref{eq2_FMe}) of Definition~\ref{def1}. The rest of the argument is the same as for the regularized system (\ref{eq2_E1}) (see Section~\ref{sec_canon} and Appendix~\ref{subsec_A2}) and leads to the relation
	\begin{equation}
	\Phi^{(N+1,\beta)} \approx \mathcal{R}[\Phi^{(N,\beta)}]
	 \quad \textrm{as} \quad N \to \infty.
	\label{eqVM_E4P}
	\end{equation}
This relation is the analog of Eq.~(\ref{eq2_E2}), but here it is valid only asymptotically for large $N$. Together with the time-scale invariance property, Eq.~(\ref{eqVM_E4P}) suggests that $\Phi^{(N,\beta)}$ are canonical regularizations related by the RG operator asymptotically for large $N$.

Assuming that the maps $\Phi^{(N,\beta)}$ belong to the basin of attraction of the fixed point (\ref{eqSRG_R2}), we have  $\Phi^{(N,\beta)} \to \Phi^{\infty}$ as $N \to \infty$ with the same limiting flow map $\Phi^\infty$ as in Section~\ref{sec_fixedpoint}. This property is verified numerically in Fig.~\ref{fig2}(a), where we plot the solution for $(N,J) = (20,1)$ from Fig.~\ref{fig1}(b) and the analogous solution for $(N,\beta) = (20,1)$. 

Next we study the validity of the correction term in Eq.~(\ref{eqRG_R3}) for the auxiliary model. For large $N$, this term competes with the contribution of the small regularization term in Eq.~(\ref{eqVM_E2}). 
The latter perturbs the RG operator in Eq.~(\ref{eqVM_E4P}); see Definition~\ref{def1}.
Using expression (\ref{eqVM_E3}) for this regularization term with $k_N = 2^N$, we estimate its magnitude for large $N$ as proportional to $k_N^{-4/3} = \rho_*^{N}$ with $\rho_* = 2^{-4/3} \approx 0.4$. 
Since $\rho_* < |\rho| = 0.5$, the regularization term decays faster than the leading eigenmode correction in Eq.~(\ref{eqRG_R3}). Hence, the leading eigenmode term in Eq.~(\ref{eqRG_R3}) remains valid. The resulting asymptotic expression reads
	\begin{equation}
	\Phi^{(N,\beta)} \approx \Phi^{\infty} + c_{\beta} \rho^N \Omega
	\ \ \textrm{as} \ \ N \to \infty,
	\label{eqVM_E4}
	\end{equation}
where the constant factor $c_{\beta}$ depends only on $\beta$. 

Expression~(\ref{eqVM_E4}) is verified numerically in Fig.~\ref{fig2}(b). Here we plot the rescaled deviation $\delta u_1(t)/c_J\rho^N$ from Fig.~\ref{fig1}(d) (bold dotted line) and the analogous rescaled deviations for the auxiliary model (\ref{eqVM_E1}) with $\beta = 1$ and $N = 15,\ldots,26$ (thin dotted and solid lines). The eigenmode prefactor in Eq.~(\ref{eqVM_E4}) is estimated as $c_{\beta = 1} \approx -0.196$. The graphs collapse at large $N$, which confirms the universality of the correction term.
The proximity of $\rho_* \approx 0.4$ to the absolute eigenvalue $|\rho| = 0.5$ explains  a rather slow convergence in Fig.~\ref{fig2}(b) for the auxiliary model. This example provides further support for our RG theory, this time for regularization without small-scale truncation.

\subsection{Viscous model and inviscid limit}

Now we are ready to explain the inviscid limit in the viscous model (\ref{eq1_V1}). Viscous regularization is not canonical, so the RG approach does not directly apply to the viscous case. However, we can relate solutions of the viscous model to solutions of the auxiliary regularized model.

Let us denote the flow map of the viscous IBVP by $\Phi^{(\nu)}$. The regularization term in the auxiliary model (\ref{eqVM_E1}) is designed such that it is identical to the viscous term in Eq.~(\ref{eq1_V1}) when
    	\begin{equation}
	\beta = \frac{\nu k_N}{|u_N(t)|}.
	\label{eqVM_2}
    	\end{equation}
This identification implies the relation between the corresponding flow maps in the form
    	\begin{equation}
	\Phi^{(\nu)}_{dt}(a,b) = \Phi^{(N,\beta)}_{dt}(a,b), \quad
	\beta = \frac{\nu k_N}{|a_N|}.
	\label{eqVM_3}
    	\end{equation}
This relation is valid only for infinitesimal time steps $dt$, because the relation (\ref{eqVM_2}) is time dependent.

Let $u^{(\nu)}(t) = \Phi^{(\nu)}_{t}(a,b)$ be the solution of the viscous IBVP for a fixed viscosity, initial and boundary conditions. Similarly to Eq.~(\ref{eqVM_3}) we write the infinitesimal time-step relation at arbitrary time $t$ as
    	\begin{equation}
	u^{(\nu)}(t+dt) = \Phi^{(\nu)}_{dt}\left(u^{(\nu)}(t),b'\right) 
	= \Phi^{(N,\beta)}_{dt}\left(u^{(\nu)}(t),b'\right), 
	\quad
	\beta = \frac{\nu k_N}{|u_N^{(\nu)}(t)|},
	\label{eqVM_3t}
    	\end{equation}
where $b'(t') = b(t+t')$ is the time-shifted boundary condition.
Let us select $N = N_t^{(\nu)}(a,b)$ as the largest shell number providing the parameter $\beta = \beta_t^{(\nu)}(a,b) = \nu k_N/|u_N^{(\nu)}(t)| \le 1$. The time dependence of such $N$ and $\beta$ for different (decreasing) viscosities $\nu$ and specific initial and boundary conditions are shown Fig.~\ref{fig2}(c,d). One can see that the inviscid limit $\nu \to 0$ implies $N \to \infty$. We argued in Section~\ref{subsec_aux} that $\Phi^{(N,\beta)} \to \Phi^\infty$ as $N \to \infty$. Assuming that this convergence is uniform with respect to $\beta \le 1$, the identity~(\ref{eqVM_3t}) implies that the evolution of the viscous solution can be approximated by the limiting flow map $\Phi^\infty$ at all times in the inviscid limit. This explains why the viscous regularization yields the same inviscid limit $\Phi^{(\nu)} \to \Phi^\infty$. This convergence is verified numerically in Fig.~\ref{fig2}(a) by comparing solutions for $\nu = 10^{-8}$ and $(N,\beta) = (20,1)$.

\subsection{Asymptotic form of deviations}\label{sec_AF}

Despite the limiting flow map remains universal, the expression (\ref{eqVM_E4}) for the universal correction term is not valid for the viscous regularization. Indeed, this term is affected by a nontrivial functional dependences of $N = N_t^{(\nu)}(a,b)$ and $\beta = \beta_t^{(\nu)}(a,b)$ in the infinitesimal identity (\ref{eqVM_3t}). As an example, the green curve in Fig.~\ref{fig2}(b) represents the rescaled deviation $\delta u_1(t)$, which is clearly different from the universal eigenmode dependence (thick dotted line). 

More information about this convergence can be obtained from the following formal derivation, which combines Eq.~(\ref{eqVM_2}) with the power-law asymptotic (\ref{eqVM_E3u}). We have $\beta \approx \nu k_N^{4/3}/|\gamma_t(a,b)|$.
Let us consider the viscosity sequence $\nu_N = k_N^{-4/3} = 2^{-4N/3}$, which vanishes as $N \to \infty$. Then, the auxiliary model parameter becomes
    	\begin{equation}
	\beta \approx \frac{1}{\gamma_t(a,b)}.
	\label{eqVM_2U}
    	\end{equation}
Note that expression (\ref{eqVM_2U}) does not depend on viscosity. 
Substituting Eq.~(\ref{eqVM_E4})  into the right-hand side of Eq.~(\ref{eqVM_3t}), we obtain
	\begin{equation}
	u^{(\nu_N)}(t+dt) 
	\approx \Phi^{\infty}_{dt}\big(u^{(\nu_N)}(t),b'\big)
	+ c_{\beta} \rho^N  \Omega_{dt}\big(u^{(\nu_N)}(t),b'\big).
	\label{eqPT_1}
    	\end{equation}
	
Let us fix the initial and boundary conditions $(a,b)$ and introduce the rescaled deviation $v(t)$ from the limiting solution $u^\infty(t) = \Phi^{\infty}_{t}(a,b)$ as 
	\begin{equation}
	u^{(\nu_N)}(t) = u^\infty(t)+\rho^N v(t). 
	\label{eqPT_1b}
    	\end{equation}
Expanding the middle term in Eq.~(\ref{eqPT_1}) to the first-order in $\rho^N$, we write
	\begin{equation}
	\Phi^{\infty}_{dt}\big(u^{(\nu_N)}(t),b'\big) 
	\approx
	\Phi^{\infty}_{dt}\big(u^{\infty}(t),b'\big) 
	+\rho^N\delta\Phi^{\infty}_{dt}\big(v(t);u^{\infty}(t),b'\big),
	\label{eqPT_1exp}
    	\end{equation}
where $\delta\Phi^{\infty}_{dt}(\delta u;u^{\infty}(t),b')$ is the variational derivative of $\Phi^\infty_{dt}(u^\infty(t)+\delta u,b')$. Similarly, the last term in Eq.~(\ref{eqPT_1}) to the first-order becomes
	\begin{equation}
	c_{\beta} \rho^N  \Omega_{dt}\big(u^{(\nu_N)}(t),b'\big)
	\approx
	c_{\beta} \rho^N  \Omega_{dt}\big(u^\infty(t),b'\big).
	\label{eqPT_1expB}
    	\end{equation}
Finally, using Eqs.~(\ref{eqPT_1b})--(\ref{eqPT_1expB}) in the relation (\ref{eqPT_1}), cancelling the zero-oder terms $u^\infty(t+dt) = \Phi^{\infty}_{dt}\big(u^{\infty}(t),b'\big)$ and then the common factor $\rho^N$, we obtain
	\begin{equation}
	v(t+dt) 
	\approx \delta\Phi^{\infty}_{dt}\left(v(t);u^{\infty}(t),b'\right)+ c_{\beta} \Omega_{dt}\big(u^\infty(t),b'\big).
	\label{eqPT_2}
    	\end{equation}
This is the asymptotic linearized equation for the rescaled correction $v(t)$.

The linearized Eq.~(\ref{eqPT_2}) must be solved with the trivial initial condition $v(0) = (0,0,\ldots)$ and $\beta$ given by Eq.~(\ref{eqVM_2U}). Note that this linearized initial value problem and, hence, its solution do not depend on viscosity $\nu_N$. Denoting this solution by $v^\infty(t)$, we write Eq.~(\ref{eqPT_1b}) as
	\begin{equation}
	u^{(\nu_N)}(t) \approx u^\infty(t)+\rho^N v^{\infty}(t)
	\quad \textrm{as} \quad \nu_N = k_N^{-4/3} \to 0.
	\label{eqPT_1R}
    	\end{equation}
This is our RG prediction for the viscous model for small viscosities. It provides the scaling of the correction term for the specific vanishing viscosity sequence. Here the functional form of the correction $v^{\infty}(t)$ is the same for all (large) $N$, but it is not universal. The universality is broken because $\beta$ in Eq.~(\ref{eqPT_2}) is given by the expression (\ref{eqVM_2U}) intrinsic to the viscous regularization. 

\begin{figure}[tp]
\centering
\includegraphics[width=0.95\textwidth]{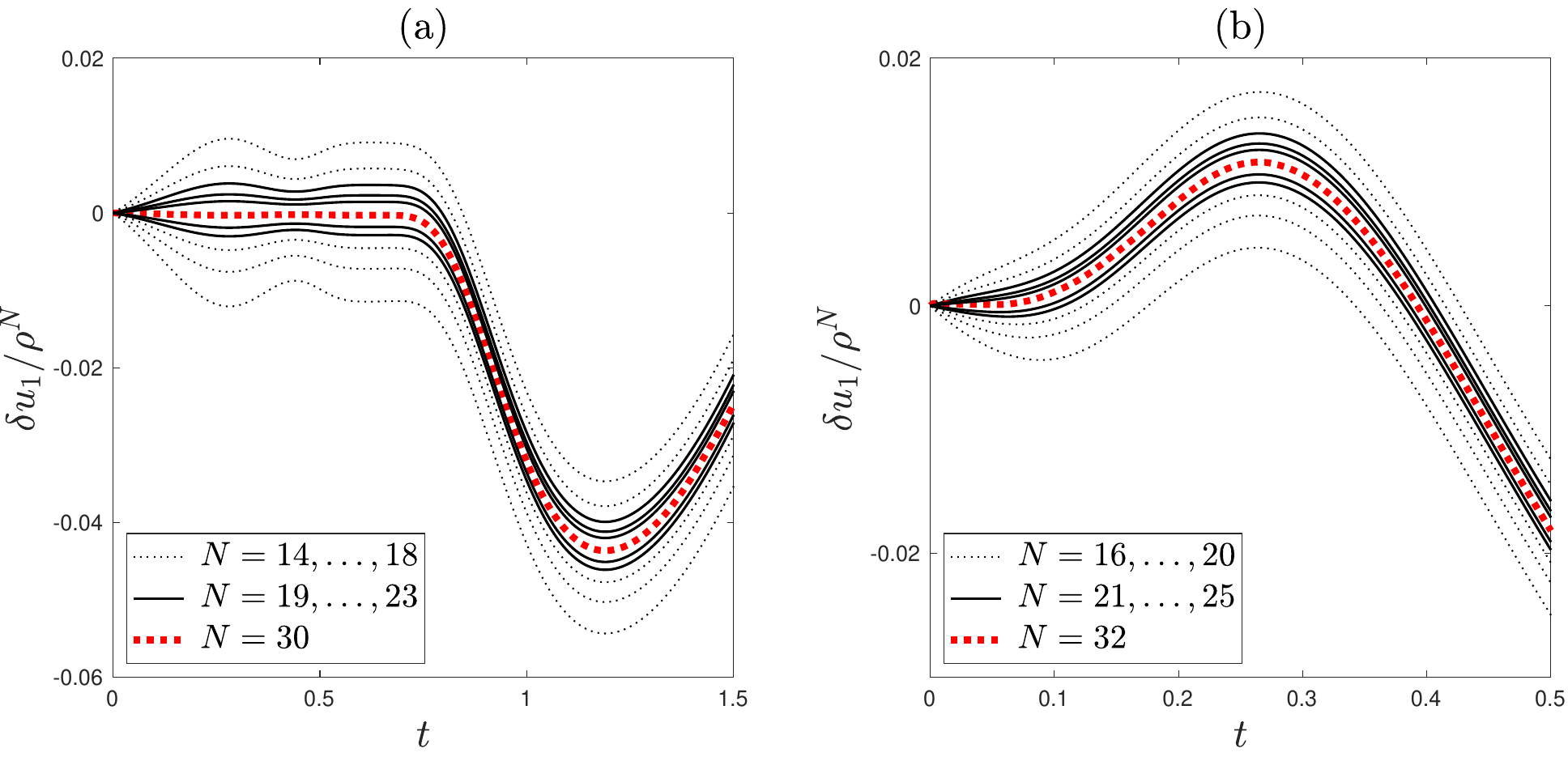}
\caption{Convergence of the rescaled deviations $\delta u_1/\rho$ in the viscous model with $\nu_N = k_N^{-4/3} = 2^{-4N/3}$ and increasing values of $N$. The panels (a) and (b) correspond to the initial conditions $\mathrm{IC}_1$ and $\mathrm{IC}_2$, respectively.}
\label{fig3}
\end{figure}

Relation (\ref{eqPT_1R}) is verified numerically in Fig.~\ref{fig3} showing the rescaled deviations $\delta u_1(t)/\rho^N$ of the first shell for the initial and boundary conditions (\ref{eqS3_2y}) and (\ref{eqS3_2bc}). The deviation is defined as $\delta u(t) = u^{(\nu_N)}(t)-u^\infty(t)$, where $u^\infty(t)$ is approximated by taking $\nu_N$ with $N = 50$. 
The rescaled deviations converge with increasing $N$, but the limiting functions are different from the universal eigenmode in Figs.~\ref{fig1}(c,d) and \ref{fig2}(b). One can also notice that the convergence rate is rather slow, similarly to Fig.~\ref{fig2}(b). This is caused by the viscous term present in Eq.~(\ref{eq1_V1}) at all (including large) scales, similarly to Eq.~(\ref{eqVM_E1}) as we explained in Section~\ref{subsec_aux}. Similar study can be carried out for other types of non-canonical (for example, hyperviscous) regularizations, which lead to different expressions for the sequence $\nu_N$ and the function $v^\infty(t)$ in Eq.~(\ref{eqPT_1R}).

In summary, the vanishing viscosity limit of the dyadic model is explained by the fixed-point RG attractor $\Phi^\infty$. Here, despite the viscous regularization is not canonical, it is related to the canonical auxiliary model by a proper time-dependent choice of parameters. In this way, the viscous regularization inherits a part (but not all) of universal properties of canonical regularizations: its inviscid limit is the same as in Eq.~(\ref{eqSRG_R2}) but lacking the universality of the correction term in Eq.~(\ref{eqRG_R3}). 
The asymptotic relation (\ref{eqPT_1R}) holds with the non-universal (intrinsic to the viscous model) correction function $v^\infty(t)$. 

\section{\newtext{Attracting closed invariant curve of RG dynamics}}
\label{sec_hopf}

In this section, we further explore the dynamical systems approach to the vanishing regularization limit. \newtext{We present here a more complex example of an RG attractor represented by a closed invariant curve.} In this case regularized solutions do not converge in the limit $N \to \infty$, but one can describe them universally in terms of a one-parameter family of flow maps.

We consider the shell model (\ref{eq1_1a}) for $\lambda = 2$, real variables $u_n \in \mathbb{R}$ and the coupling function 
     	\begin{equation}
	f_n := f(u_{n-2},\ldots,u_{n+2}) = 
	\frac{9}{40}\,u_{n-1}u_{n-2} 
	+\frac{11}{20}\,u_{n+1}u_{n-1}  
	-2 u_{n+2}u_{n+1}
	 +2u_{n+1}^2-u_nu_{n-1}.
	\label{eqSM_2G}
    	\end{equation}
The function (\ref{eqSM_2G}) is designed as a combination of Gledzer's models~\cite{gledzer1973system} with energy-conserving nonlinearity, and it couples each shell variable to two neighbors from each side. 
	
The ideal IBVP is defined by Eqs.~(\ref{eq1_1a}) and (\ref{eqSM_2G}) with the initial conditions (\ref{eq1_IC}) and the two boundary conditions
    	\begin{equation}
	u_{-1}(t) = b_{-1}(t),\ \ u_0(t) = b_0(t) \ \ \textrm{for} \ \ t \ge 0.
	\label{eq1_BCG}
    	\end{equation}
We will use the short notation $b = (b_{-1},b_0)$ for a pair of continuously differentiable boundary functions $b_{-1},b_0 \in C^1$. The viscous model takes the form (\ref{eq1_V1}).

\subsection{RG formalism}

The RG formalism is defined by the symmetries and couplings of the ideal IBVP. Similarly to the dyadic model, the symmetries are the time and space scalings.
The time scaling symmetry has the same form (\ref{eq1_2TS}). For the boundary condition (\ref{eq1_BCG}), the space scaling symmetry is formulated as
     	\begin{equation}
	\tilde{u}(t) = S^+u(t), 
	\quad \tilde{a} = S^+ a, 
	\quad \tilde{b}(t) = (\lambda b_0(t), \, \lambda u_1(t)).
	\label{eq1_2SS_G}
    	\end{equation}
The RG formalism is introduced in the same way as in Section~\ref{sec_canon}; see Definitions~\ref{def1} and \ref{def2}. 
\newtext{Here Eqs.~(\ref{eq2_FMe}) and (\ref{eq2_FMf}) are adapted to coupling function (\ref{eqSM_2G}) and boundary condition (\ref{eq1_BCG}) as
	\begin{equation}
    	\frac{du_1}{dt} = k_1 f(b_1,b_0,u_1,u_2,u_3), \quad u_1(0) = a_1,
	\label{eq2_FMeN}
    	\end{equation}
where $u_2(t) = \tilde{u}_1(t)/\lambda$ and $u_3(t) = \tilde{u}_2(t)/\lambda$ with 
	\begin{equation}
	\tilde u(t) = \Phi_t(\tilde a,\tilde b), \quad
 	\tilde a = S^+a; \quad
	\tilde b(t) = \big(\lambda b_0(t),\,\lambda u_1(t)\big).
 	\label{eq2_FMfG}
    	\end{equation}
}The RG operator $\mathcal{R}$ acts in the space of canonical regularizations. Canonical regularizations are flow maps $\Phi$, which are infinitely renormalizable and invariant with respect to the time scaling.
Since the RG operator maps $N$-level to $(N+1)$-level canonical regularizations, the vanishing regularization limit is associated with the RG dynamics $\mathcal{R}^N[\Phi]$ as $N \to \infty$. 
Recall that the RG formalism separates the roles of the ideal system and regularization:
the RG operator depends on the coupling function (\ref{eqSM_2G}) of the ideal model, while all information on regularization is contained in the flow map $\Phi$. 

Examples of canonical regularizations are given by the models (\ref{eq2_E1}), where $u_n(t) = 0$ for $n > N+J$ and $t > 0$.
We denote by $\Phi^{(N,J)}$ a flow map of the respective regularized IBVP.
Like in Section~\ref{sec_canon}, these flow maps satisfy the RG relation: $\Phi^{(N+1,J)} = \mathcal{R}[\Phi^{(N,J)}]$.
Models (\ref{eq2_E1}) provide a class of canonical regularizations that we use for numerical analysis of the RG dynamics.

\subsection{RG attractor}

Figures \ref{figQP1}(a,b) show evolutions of shell variables $u_n(t)$ of the regularized model (\ref{eq2_E1}) and (\ref{eqSM_2G}) with $(N,J) = (20,3)$ for two initial conditions (\ref{eqS3_2y}) and boundary conditions $b_{-1}(t) = 1$ and $b_0(t) = 2+\sin t$. Figures \ref{figQP1}(c,d) show the respective evolutions of the first two variables $(u_1,u_2)$ for different cutoff parameters $N = 30,\ldots,50$. For the initial condition $\mathrm{IC}_1$ (panel c), the figures verify that the regularized solutions converge at pre-blowup times $t \le t_b \approx 3.63$ but diverge at larger times $t > t_b$. For the initial condition $\mathrm{IC}_2$ (panel d), the regularized solutions diverge at all positive times.

\begin{figure}[tp]
\centering
\includegraphics[width=0.98\textwidth]{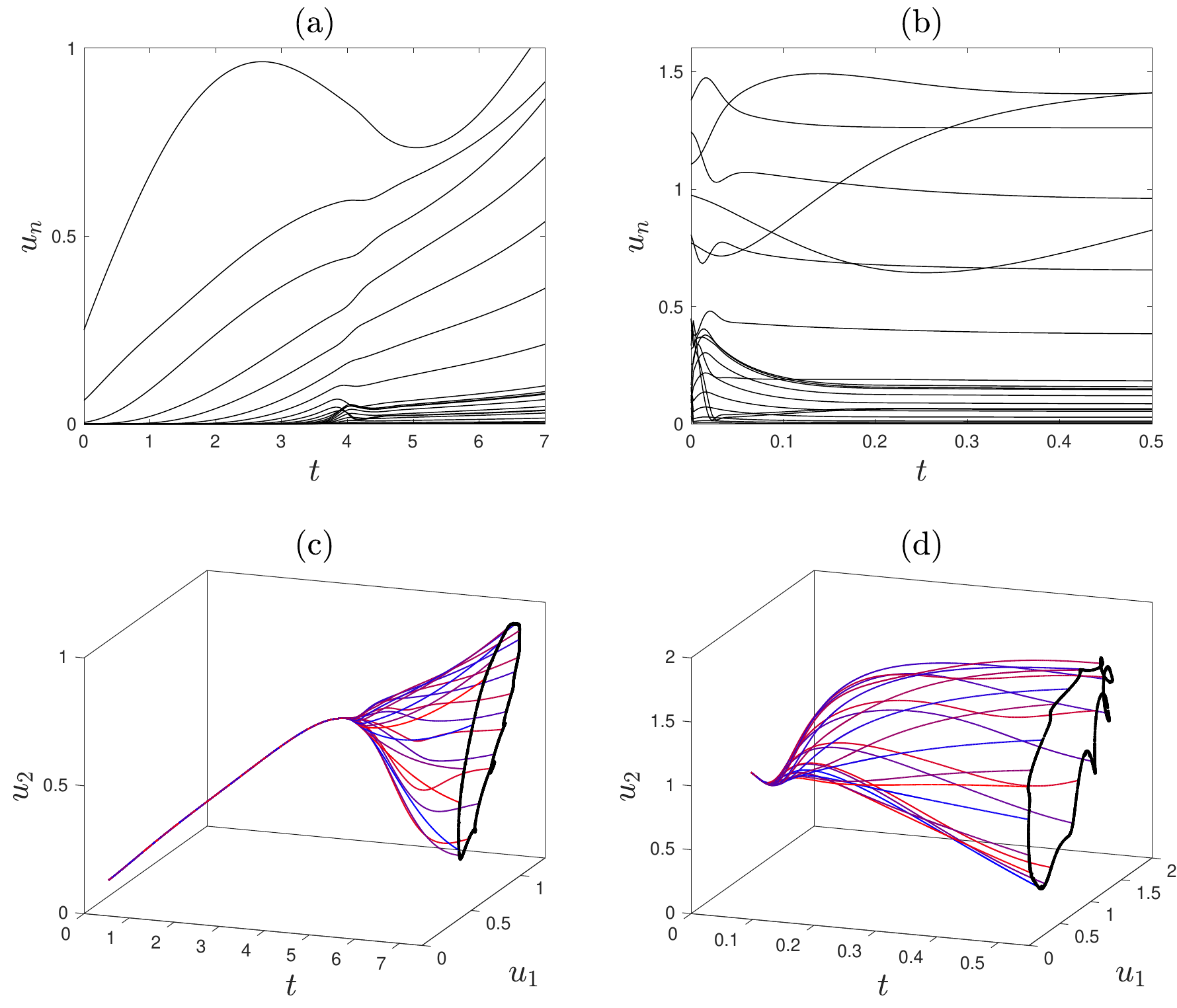}
\caption{Evolutions of all shell variables $u_n(t)$ for the regularized model (\ref{eq2_E1}) and (\ref{eqSM_2G}) with $(N,J) = (20,3)$ and the initial condition (a) $\mathrm{IC}_1$ and (b) $\mathrm{IC}_2$. Lower panels show the respective evolutions of the first two variables $(u_1,u_2)$, where the cutoff parameter is changed in the range $N = 30,\ldots,50$. Here the color changes gradually from red for $N = 30$ to blue for $N = 50$.}
\label{figQP1}
\end{figure}

Empty dots in Fig.~\ref{figQP2} represent the values of $(u_1,u_2)$ at the fixed times: $t = 7$ for the first and $t = 0.5$ for the second initial condition; these are terminal points of solutions from Figs.~\ref{figQP1}(c,d). Clearly, the RG attractor is not a fixed point in this case. We cannot, however, perform our numerical simulations for much larger numbers $N$. Therefore, we employ a different strategy for the study of the RG attractor: we will analyze the RG dynamics $\mathcal{R}^N[\Phi]$ starting from different canonical flow maps $\Phi$, where the latter are constructed randomly as follows.

Let us consider regularizations (\ref{eq2_E1}) with $J = 3$ and modified dissipative terms $c_n|u_n| u_n$, which are multiplied by positive coefficients $c_n$. We perform $100$ simulations for each $N = 40,\ldots,80$ by choosing $c_n$ randomly from the interval $[0,3]$. 
The resulting values $(u_1,u_2)$ are presented by black dots in Fig.~\ref{figQP2}, which form closed curves in the phase space. According to Eq.~(\ref{eq2_E2}), each random choice of regularization corresponds to a specific RG sequence $\mathcal{R}^N[\Phi]$ starting from a different initial flow map $\Phi$. Hence, the black dots forming a curve in Fig.~\ref{figQP2} determine the RG attractor probed by 100 different initial conditions. Similar curves appear for any pair of shell variables. The one-dimensional asymptotic structure is also apparent in Figs.~\ref{figQP1}(c,d), where the data from Fig.~\ref{figQP2} is added at final times. 

\begin{figure}[tp]
\centering
\includegraphics[width=0.9\textwidth]{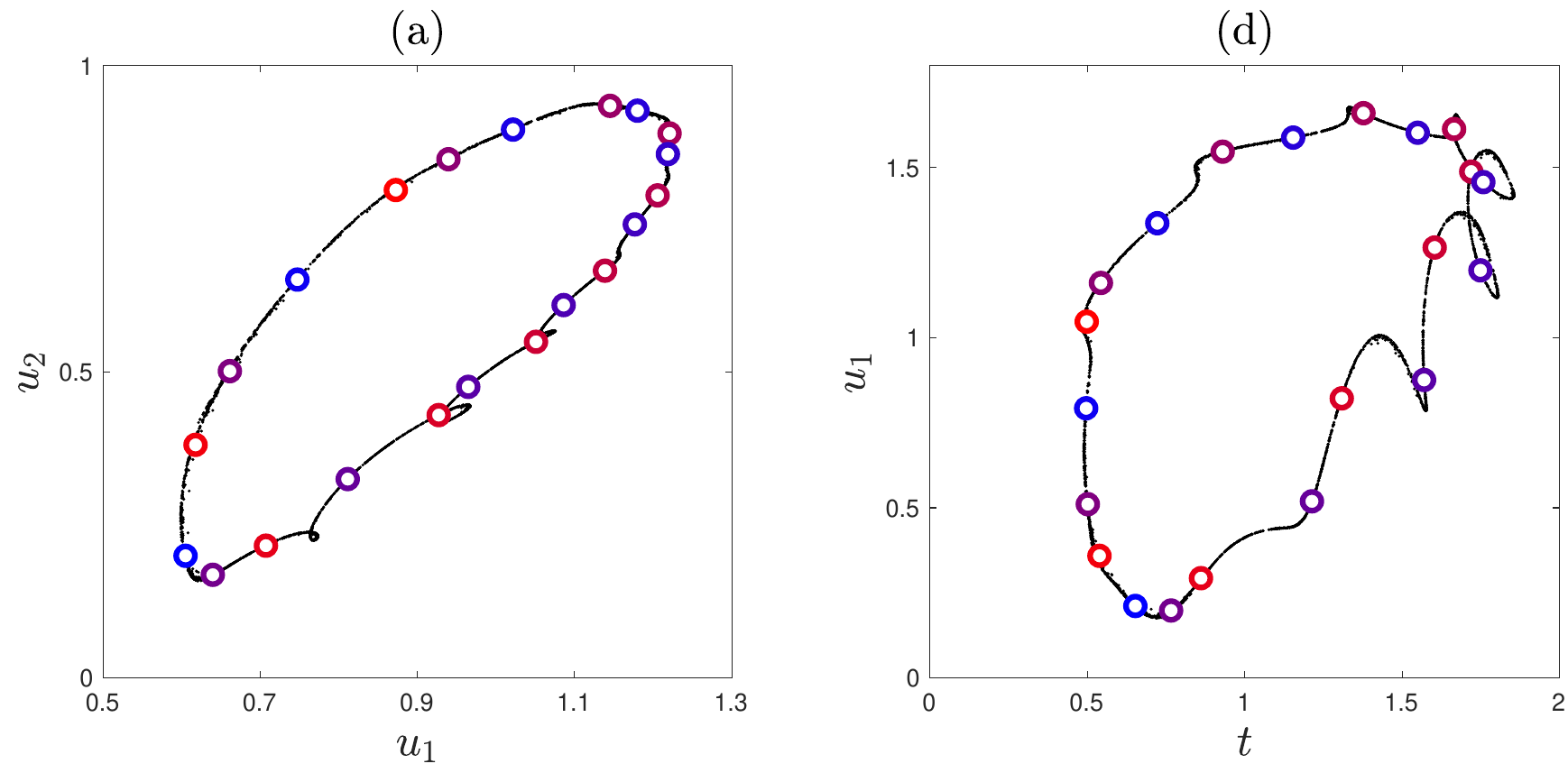}
\caption{Evolutions of the first two variables $(u_1,u_2)$ for the initial conditions and times: (a) $\mathrm{IC}_1$ and $t = 7$, (b) $\mathrm{IC}_2$ and $t = 0.5$.
Empty dots correspond to the regularizations (\ref{eq2_E1}) with $J = 3$, where the cutoff parameter is changed in the range $N = 30,\ldots,50$.
Here the color varies gradually from red for $N = 30$ to blue for $N = 50$.
Black dots forming a closed curve are obtained using 100 random canonical regularizations with $N = 40,\ldots,80$.}
\label{figQP2}
\end{figure}

\newtext{Our numerical results suggest that the attractor of the RG dynamics is a closed curve $\Phi^{(\theta)}$ in the space of flow maps parametrized by points of a circle, $\theta \in S^1$. In the theory of dynamical systems, such attractors appear, e.g., in the Hopf bifurcation of maps~\cite{guckenheimer2013nonlinear}. Note also that the folds and self-intersections of the invariant curves in Fig.~\ref{figQP2} are a consequence of the two-dimensional projection; examination of the three-dimensional graphs in the space $(u_1,u_2,u_3)$ does not reveal any singularities.}
This attractor leads to a more sophisticated (than in the fixed-point case) but still universal description of the vanishing regularization limit. Namely, within the basin of attraction, $\Phi \in \mathcal{B}(\Phi^{(\theta)})$, the RG dynamics takes the asymptotic form
	\begin{equation}
	\mathcal{R}^N[\Phi] \approx \Phi^{(\theta_N)}
	 \quad \textrm{as} \quad N \to \infty,
 	\label{eqVQP1}
    	\end{equation}
where $\Phi^{(\theta_N)}$ are flow maps of the attracting invariant curve. The phases $\theta_N$ depend on the initial flow map $\Phi$, and are related by the circle map $\mathcal{R}_\theta: \theta_{N} \mapsto \theta_{N+1}$ induced by the RG dynamics on the attractor. In other words, for any canonical regularization from the basin of attraction, regularized solutions $u(t) = \mathcal{R}^N[\Phi]_t(a,b)$ approach asymptotically a universal (i.e., independent of regularization) family of solutions $\Phi^{(\theta)}_t(a,b)$. The latter is a one-parameter family of solutions of the ideal IBVP. 

\subsection{Inviscid limit}\label{QPinv}

The viscous regularization (\ref{eq1_V1}) is not canonical because it is not time-scale invariant.
As in Section~\ref{subsec_visc}, the inviscid limit can be related to the RG dynamics of the canonical auxiliary regularization (\ref{eqVM_E1}): the corresponding flow maps are related at infinitesimal time steps by Eq.~(\ref{eqVM_3}). In this relation, $\Phi^{(\nu)}_{dt}(a,b)$ is the flow map of the viscous IBVP, and $\Phi^{(N,\beta)}_{dt}(a,b)$ is the flow map of the auxiliary IBVP with properly chosen parameters $N = N_t^{(\nu)}(a,b)$ and $\beta = \beta_t^{(\nu)}(a,b)$. Verifying numerically that the limit $\nu \to 0$ corresponds to $N \to \infty$ with $\beta \le 1$ and using Eq.~(\ref{eqVQP1}), one concludes that
	\begin{equation}
	\Phi^{(\nu)}_{dt}(a,b) \approx \Phi^{(\theta)}_{dt}(a,b) \quad \textrm{as} \quad \nu \to 0.
 	\label{eqVQP2}
    	\end{equation}
Here the parameter $\theta$, which selects a map from the RG attractor, depends on $(N,\beta)$, which, in turn, are functions of viscosity, time, initial and boundary conditions. This dependence implies that the universal infinitesimal-time relation (\ref{eqVQP2}) does not extend to finite times, unlike in the case of canonical regularizations (\ref{eqVQP1}).

\begin{figure}[tp]
\centering
\includegraphics[width=0.9\textwidth]{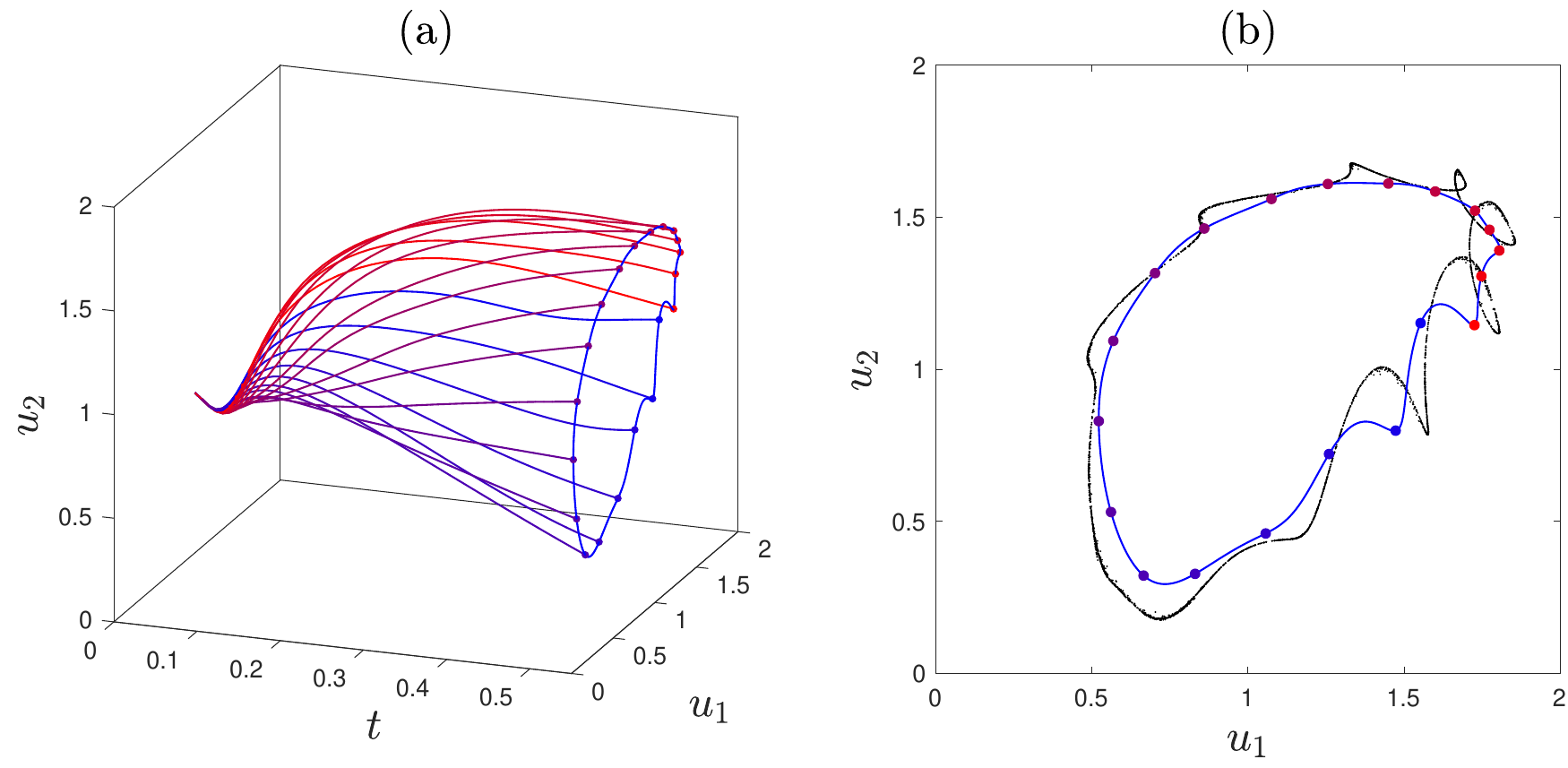}
\caption{(a) Evolutions of the first two variables $(u_1,u_2)$ for the viscous model with the second initial condition $\mathrm{IC}_2$. Different curves correspond to different viscosities logarithmically spaced in the interval $1.65 \times 10^{-10} \le \nu \le 10^{-6}$. (b) Terminal points (at time $t = 0.5$) of the viscous solutions from the left panel are compared with the data from Fig.~\ref{figQP2}(b) featuring the closed-curve RG attractor.}
\label{figQP3}
\end{figure}

The breaking of universality by the viscous regularization is demonstrated in Fig.~\ref{figQP3}(a), where we plot the evolutions of variables $(u_1,u_2)$ for different viscosities logarithmically spaced in the interval $1.65 \times 10^{-10} \le \nu \le 10^{-6}$. Figure~\ref{figQP3}(b) verifies that the viscous solutions are different from the ones corresponding to the RG attractor. Still, the viscous solutions form a closed one-parameter surface with a change of the viscous parameter like in Figs.~\ref{figQP1}(a,b). We expect that this behavior can be explained as a consequence of the small-scale self-similarity, as we did in Section~\ref{sec_AF}. This analysis would be based on a relation similar to Eq.~(\ref{eqVM_E3u}), which may gain a quasiperiodic dependence on $n$. We will not continue this investigation here, leaving it for future studies; note that a similar analysis was performed in \cite{mailybaev2016spontaneous}.

In summary, we see that the RG dynamics can be characterized by different types of attractors that determine universal properties of solutions in the limit of vanishing regularization. One can also study a dependence of the RG dynamics by introducing a parameter in the ideal system. Such a parametric analysis would link qualitative changes in the inviscid limit with bifurcations of the RG attractors. The next example demonstrates the chaotic behavior, the consequences of which go further and can be associated with the phenomenon of spontaneous stochasticity.

\section{Chaotic RG dynamics in the Sabra model}
\label{sec_sabra}

In this section, we extend the RG formalism to the Sabra model~\cite{l1998improved}. 
Equations of the ideal Sabra model are written for $\lambda = 2$ and complex variables $u_n \in \mathbb{C}$ in the form (\ref{eq1_1a})
with the coupling function 
     	\begin{equation}
	f_n := f(u_{n-2},\ldots,u_{n+2}) = i\left(
	\frac{u_{n-1}u_{n-2}}{4} 
	-\frac{u_{n+1}u_{n-1}^*}{2}  
	+2 u_{n+2}u_{n+1}^*
	 \right),
	\label{eqSM_2}
    	\end{equation}
where $i$ is the imaginary unit and the stars denote complex conjugation. The ideal IBVP is defined by Eqs.~(\ref{eq1_1a}) with the initial conditions (\ref{eq1_IC}) and the boundary conditions (\ref{eq1_BCG}).
Solutions of the ideal IBVP blowup in a finite time~\cite{dombre1998intermittency,mailybaev2012renormalization} and require regularization for extending beyond the blowup time. 
The viscous model takes the form (\ref{eq1_V1}).

The ideal IBVP has the time scaling symmetry (\ref{eq1_2TS}) and the space scaling symmetry 
(\ref{eq1_2SS_G}). There is an additional phase symmetry 
     	\begin{equation}
	\tilde{u}_n(t) =  e^{iF_n} u_{n}(t), 
	\quad \tilde{a}_n = e^{iF_n} a_{n}, 
	\quad \tilde{b}(t) = (e^{iF_{-1}} b_{-1}(t), \, e^{iF_0} b_0(t)),
	\label{eqSM_3}
    	\end{equation}
where $F = (F_n)_{n \in \mathbb{Z}}$ is an arbitrary Fibonacci sequence, i.e., $F_{n} = F_{n-1}+F_{n-2}$ for $n \ge 1$. This symmetry mimics translations in physical space, recalling that these translations are given by phase factors in the Fourier representation. 

The RG formalism for the Sabra model is developed in the same way as in Sections~\ref{sec_canon} and \ref{sec_hopf}. 
We learned in Section~\ref{sec_canon} that the RG operator is defined unambiguously if it is restricted to canonical flow maps, which are maximally symmetric. The latter means that the flow maps are invariant with respect to all symmetries except the space scaling (\ref{eq1_2SS_G}). \newtext{As observed in the example of time-scale invariance broken by viscous regularization, this introduces ambiguity in the definition of the RG operator (see Section~\ref{sec_canon}) and also affects the universality of the limiting behavior (see Section~\ref{subsec_visc}). We expect that these features may extend to any additional symmetry of the system.
In the Sabra model, such additional symmetries are the time scaling and phase symmetries.} From now on we assume that all regularizations under consideration are invariant with respect to the phase symmetry (\ref{eqSM_3}). For a flow map $\Phi$ this condition means that the regularized solutions $u(t) = \Phi_t(a,b)$ and $\tilde u(t) = \Phi_t(\tilde a,\tilde b)$ are related by Eq.~(\ref{eqSM_3}). One can see that the viscous model (\ref{eq1_V1}) is phase invariant, but not invariant to the time scaling. 

\begin{figure}[tp]
\centering
\includegraphics[width=0.9\textwidth]{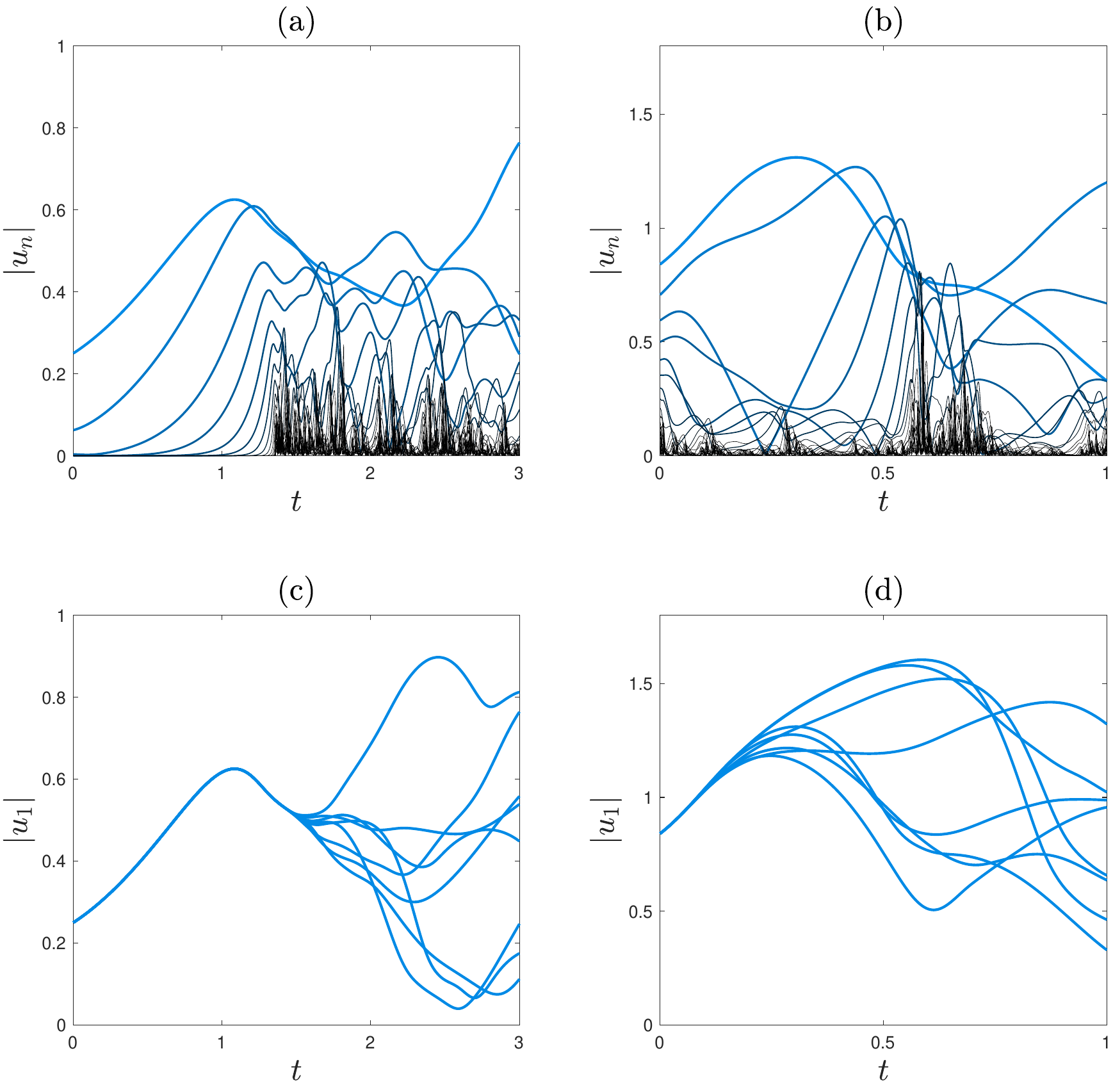}
\caption{Absolute values of shell variables $|u_n(t)|$ for the regularized Sabra model with $(N,J) = (17,2)$ and the initial conditions (a) $\mathrm{IC}_1$ and (b) $\mathrm{IC}_2$. Larger $n$ correspond to darker and thiner curves. The panels (c) and (d) correspond to the same initial conditions and show the absolute variables $|u_1(t)|$ for different $N = 10,11,\ldots,17$.}
\label{fig5}
\end{figure}

The rest of this section presents the results of numerical simulations for the canonical regularizations (\ref{eq2_E1}) with $J = 2$ and different $N$. 
We consider two different (regular and rough) initial conditions of the form
    	\begin{equation}
    	\label{eq3_IC1}
	\mathrm{IC}_1: \ \ a_n = 2^{-k_n}e^{i\sqrt{n}}, \quad 
	\mathrm{IC}_2: \ \ a_n = k_n^{-1/4}e^{i\sqrt{n}}.
    	\end{equation}
and the boundary conditions
 	\begin{equation}
	\label{eq3_IC1B}
	b_{-1}(t) = 1/2, \quad b_0(t) = e^{-it}.
    	\end{equation}
The corresponding evolutions of absolute values $|u_n(t)|$ are presented in Figs.~\ref{fig5}(a,b) for $N = 17$. They demonstrate sharp intermittent fluctuations at small scales strikingly different from mild oscillations in the dyadic model in Figs.~\ref{fig1}(a,b). Figures~\ref{fig5} (c,d) show the time evolutions of the absolute variables $|u_1(t)|$ corresponding to the shell $n = 1$ (the largest scale of motion). Here different curves correspond to different $N = 10,11,\ldots,17$. These graphs suggest that the RG dynamics does not possess a regular (e.g., fixed-point) attractor. We remark that the divergence of solutions for the regular initial condition occurs after the blowup time $t_b \approx 1.36$; see Figs.~\ref{fig5}(a,c).

The irregular change of solutions with increasing $N$ is a hint that the RG dynamics may be chaotic. In order to test this hypothesis, we study the growth of small disturbances of the flow map in the RG dynamics. In the classical chaos, one would observe an exponential growth as a consequence of a positive Lyapunov exponent. 
Let us introduce a sequence of perturbed flow maps $\Phi^{(N,J,\varepsilon)}$ by replacing the term $|u_n|u_n$ in Eq.~(\ref{eq2_E1}) with $(1+\varepsilon)|u_n| u_n$, where $\varepsilon$ is a small perturbation parameter. One can see that $\Phi^{(N,J)} = \Phi^{(N,J,\varepsilon = 0)}$ and the perturbed flow maps satisfy the RG relation analogous to Eq.~(\ref{eq2_E2}) as 
	\begin{equation}
	\Phi^{(N+1,J,\varepsilon)} = \mathcal{R}[\Phi^{(N,J,\varepsilon)}].
	\label{eqSM_5}
    	\end{equation}
Hence, we can analyze the growth of perturbations in the RG dynamics by looking at the difference 
	\begin{equation}
	\delta u(t) = \Phi_t^{(N,J,\varepsilon)}(a,b)-\Phi_t^{(N,J)}(a,b) 
	\label{eqSM_5D}
    	\end{equation}
with increasing $N$ while keeping $J$ and $\varepsilon$ fixed. In this analysis, $\varepsilon$ controls the size of initial perturbation in the space of canonical flow maps.

We set the very small value of $\varepsilon = 10^{-13}$ and compute the perturbations $\delta u(t)$ with very high accuracy for the same initial and boundary conditions (\ref{eq3_IC1}) and (\ref{eq3_IC1B}). Magnitudes of the perturbations are measured with the energy norms $\|\delta u(t) \| = \left( \sum_n |\delta u_n(t)|^2 \right)^{1/2}$. The results are presented in Fig.~\ref{fig6} for $N = 1,\ldots,15$, where we plot $\|\delta u(t) \|$ at $t = 3$ for the first and at $t = 1$ for the second initial condition. The main plot is given in the logarithmic vertical scale demonstrating that the separation between the flow maps grows faster than exponentially. This indicates that the RG dynamics is indeed chaotic, though the separation of solution is faster than exponential as in classical chaotic systems. 

\begin{figure}[t]
\centering
\includegraphics[width=0.56\textwidth]{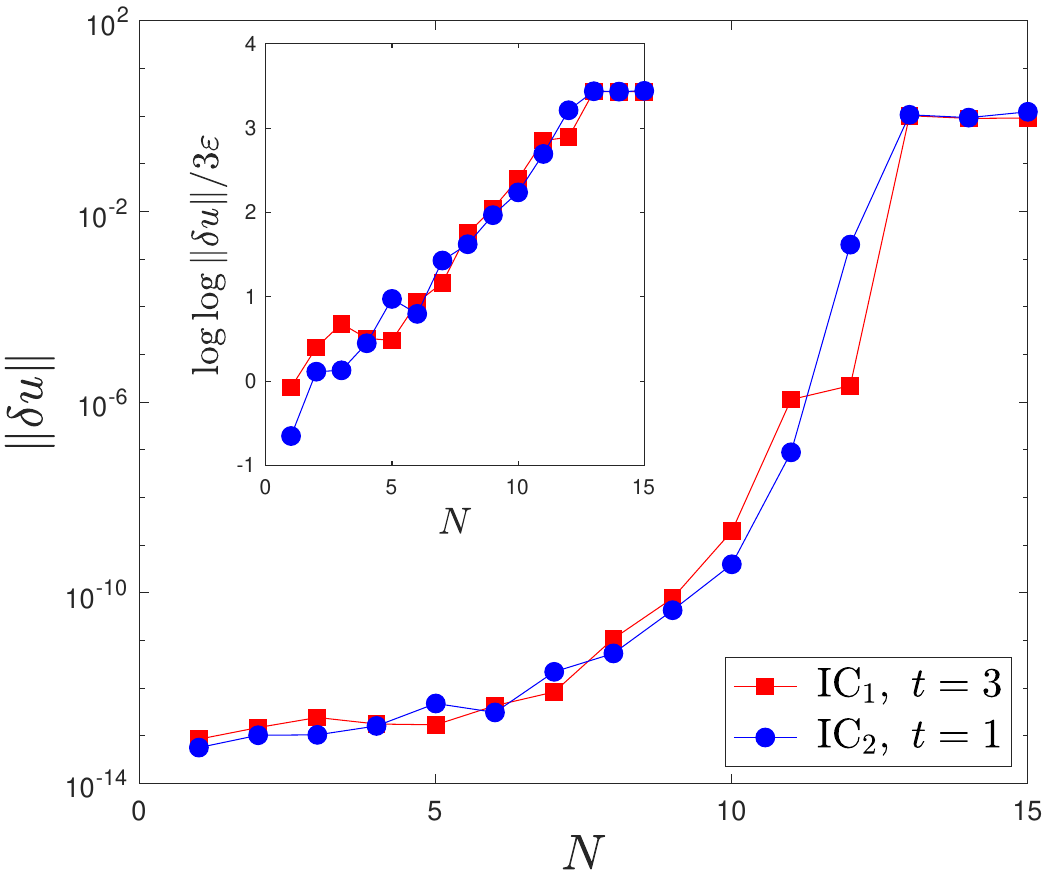}
\caption{The super-exponential growth of small perturbations (\ref{eqSM_5D}) in the RG dynamics of the Sabra model. The two graphs represent the dependence on $N$ for the deviations $\|\delta u(t)\|$ taken at $t = 3$ for the first and $t = 1$ for the second initial conditions. The main plot is given in vertical logarithmic scale. The inset suggest the double exponential growth of perturbations by plotting $\log \log \left(\|\delta u(t) \|/3\varepsilon\right)$ for different $N$.}
\label{fig6}
\end{figure}

The inset of Fig.~\ref{fig6} shows the graphs of $\log \log \left(\|\delta u(t) \|/3\varepsilon\right)$ as functions of $N$, proposing that the growth of perturbations in the RG dynamics is double exponential in $N$. Since $l_N = 2^{-N}$, such a double exponential function of $N$ would be exponential in a negative power of the regularization scale $l_N$. This relation connects our observations to earlier phenomenological results in the turbulence theory relating positive Lyapunov exponents with a dissipative scale; see e.g.~\cite{ruelle1979microscopic}. \newtext{A similar double-exponential growth of perturbations in the RG dynamics was also observed in discrete-time models~\cite{mailybaev2023spontaneously,mailybaev_RG_2025}.}

\section{Discussion}
\label{sec_disc}

Typically, scale-invariant physical systems require regularization at small and/or large scales for determining well-posed long-time solutions. 
What makes such solutions converge (or not) in the limit when the regularization vanishes?
Even for such a charismatic example as the Burgers equation, the rigorous theory is not developed for many regularizations, e.g., the hyperviscous ones.
At the same time numerical simulations suggest that different dissipative mechanisms yield just the same limiting solutions.
We want to understand how the convergence mechanism is related to the multi-scale nature of a system, and how it depends on the choice of regularization. Being interested in  qualitative aspects of the regularization process, we focused our research on shell models, which are much simpler (but still quite complex) multi-scale toy models of  realistic physical systems.

We study initial boundary value problems (IBVP) for ideal and regularized shell models.
In general, each regularized IBVP can be represented by a flow map that contains solutions for all initial and boundary conditions.
We show how the scale invariance of an ideal system defines a renormalization group (RG) operator in the space of flow maps. 
This operator acts by biasing the regularization towards smaller scales and introducing ideal interactions at the freed largest scale.
In this way, we establish the correspondence between the vanishing regularization limit and the RG dynamics in the space of flow maps. 
Thereby, the limiting flow map solving the ideal IBVP is associated with the RG attractor, and the convergence is controlled asymptotically by the leading eigenmode of the linearized RG operator. 
Since the RG operator is defined by the ideal system (with all information on a specific regularization contained in flow maps), this yields a qualitative explanation of the universality in the regularization process.
This method is first applied to the dyadic shell model, where the RG dynamics has a fixed-point attractor.
Using the Gledzer-type shell model, we demonstrated a different situation, when the RG attractor is represented by an invariant closed curve. This example reveals a more sophisticated but still universal behavior in the limit of vanishing regularization. \newtext{The RG attractor represented by an invariant closed curve may also be related to the phenomenon of spontaneous stochasticity when small-scale noise is added to the equations of motion, as studied previously in a similar Gledzer-type model \cite{mailybaev2016spontaneous}.}

Compared to existing RG theories, our approach is closer in spirit to the Feigenbaum theory of a period-doubling cascade~\cite{feigenbaum1983universal}, which operates with the space-time scaling of dynamical systems represented by maps. 
Unlike the Wilson theory of phase transition~\cite{wilson1983renormalization} that introduces coarse-graining at small scales, small details of the flow map are accurately preserved in our RG transformation, and the same refers to the large scales as well. Nevertheless, the role of the RG attractor and leading eigenmode in explaining universality is similar to all theories mentioned. In the dyadic model, the fixed-point RG attractor defines the universal limiting solutions, and the leading eigenmode predicts the universality of deviations from the limiting solutions.

Another nontrivial aspect is that the domain of the RG operator is determined by the symmetries of the ideal system. We have shown that a well-defined RG operator acts in the space of canonical regularizations. These are regularizations with the maximum degree of symmetry: all symmetries must be preserved, except for spatial scaling. Physically motivated regularizations are not necessarily canonical: for example, viscous regularization breaks the time-scale invariance. We showed that an explanation of the inviscid limit in terms of the RG dynamics is still possible, but at the expense of loosing some of the universal properties. In particular, the inviscid limit yields the same universal solutions for the viscous dyadic model, while the deviations cease to be universal. This emphasizes the exceptional role of symmetries in regularized dynamics.

The developed RG approach is potentially applicable to realistic models, such as the Burgers and Navier-Stokes equations. Although obtaining any rigorous results using the RG approach would be very difficult, its predictions are non-trivial and can be verified numerically. This refers both to the universal limiting solutions and universal form of deviations, as we demonstrated with the dyadic model. 
We emphasize that our RG approach is qualitative and aims to explain the universal features of regularizations rather than to calculate specific numerical values.
Another promising direction is the (potentially rigorous) application of our RG formalism to even simpler toy models such as fractal lattice models~\cite{mailybaev2023spontaneously,mailybaev2023spontaneous,mailybaev_RG_2025} and point singularities in low-dimensional ordinary differential equations~\cite{drivas2021life,drivas2020statistical,eyink2020renormalization}.
It would also be interesting to understand the relation of our RG formalism to the hidden scale invariance formulated both for shell models~\cite{mailybaev2021hidden} and Navier--Stokes system~\cite{mailybaev2020hidden,mailybaev2022hidden}. In particular, the extended form of the hidden symmetry~\cite{mailybaev2023hidden,magacho2025scale} is limited to canonical regularizations for similar reasons, predicting the hidden self-similar statistics at small scales in both inertial and dissipative ranges. This analysis may be useful for extending the RG approach to non-canonical (but physically motivated) regularizations, following the auxiliary constructions of Sections~\ref{subsec_visc} and \ref{QPinv}.

In the final part of the paper we tested the Sabra shell model. This model is known to be spontaneously stochastic in the formulation that includes both viscosity and small-scale noise~\cite{mailybaev2016spontaneously,bandak2024spontaneous}. The spontaneous stochasticity means that solutions remain stochastic in the limit when both viscous and noise terms vanish; similar behavior was reported in other shell models as well~\cite{mailybaev2016spontaneous,mailybaev2017toward,biferale2018rayleigh}. In this paper, we demonstrated that the RG dynamics of the Sabra model is chaotic. These two properties, the spontaneous stochasticity and the chaotic RG attractor, are related as we demonstrated in \cite{mailybaev2023spontaneous,mailybaev_RG_2025} for discrete space-time models. For the Sabra model, this relation will be studied in the forthcoming paper.

\section{Appendix}

\subsection{Well-posedness of the regularized IBVP}
\label{subsec_A1}

Consider the regularized Eqs.~(\ref{eq2_E1}) with $u_n(t) = 0$ for $n > N+J$ and the dyadic function (\ref{eq1_1fn}). The corresponding IBVP for initial conditions (\ref{eq1_IC}) and boundary condition (\ref{eq1_BC}) reduces to the initial-value problem for the system of $N+J$ ordinary differential equations. Since the functions (\ref{eq1_1fn}) are smooth, the classical theory of ordinary differential equations tells that this problem is well-posed locally in time. For extending this result to arbitrary positive time, it is sufficient to show that the variables $u_n(t)$ at shells $n = 1,\ldots,N+J$ remain finite (do not blowup) at all times.

Differentiating the norm $\|u\| = \left(\sum_{n =1}^{N+J} u_n^2\right)^{1/2}$ with respect to time and using Eq.~(\ref{eq2_E1}), one has
    	\begin{equation}
	\frac{d\|u\|}{dt} 
	= \frac{1}{\|u\|} \sum_{n =1}^{N+J} u_n \frac{du_n}{dt}
	= \frac{1}{\|u\|} \sum_{n =1}^{N+J} k_nu_n f_n
	-\frac{1}{\|u\|} \sum_{n =N+1}^{N+J} k_n|u_n|^3.
	\label{eqAE_1a}
	\end{equation}
Dropping the dissipative terms yields the inequality
    	\begin{equation}
	\frac{d\|u\|}{dt} 
	\le \frac{1}{\|u\|} \sum_{n =1}^{N+J} k_nu_n f_n.
	\label{eqAE_1b}
	\end{equation}
Substituting Eq.~(\ref{eq1_1fn}) with boundary condition (\ref{eq1_BC}) and $u_n(t) = 0$ for $n > N+J$ into Eq.~(\ref{eqAE_1b}), after proper cancelations yields
    	\begin{equation}
	\frac{d\|u\|}{dt} \le \frac{k_1b_0^2u_1}{\|u\|} \le k_1b_0^2,
	\label{eqAE_1}
	\end{equation}
where the last inequality follows from the property $u_1 \le \|u\|$. Integrating Eq.~(\ref{eqAE_1}), we have
    	\begin{equation}
	0 < \|u(t)\| < \|u(0)\|+k_1 \int_0^t b_0^2(t') dt' < \infty,
	\label{eqAE_2}
	\end{equation}
proving that $\|u\|$ and, hence, all shell variables remain finite at all times.

\subsection{Canonical property and RG relation for the regularized IBVP}
\label{subsec_A2}

The well-posedness of the regularized IBVP shown in the previous subsection implies the existence and uniqueness of the flow maps $\Phi^{(N,J)}$. Equations (\ref{eq2_E1}) with the functions (\ref{eq1_1fn}) are invariant with respect to the time scaling (\ref{eq1_2TS}). Hence, the flow maps are also time-scale invariant. It remains to prove the RG relation (\ref{eq2_E2}) using Definition~\ref{def1}. 

Let $u(t) = \Phi_t^{(N+1,J)}(a,b)$ with $N \ge 0$. Since $N+1 \ge 1$, the first equation in (\ref{eq2_E1}) coincides with the first equation in (\ref{eq2_FMe}). The remaining relations of Eq.~(\ref{eq2_FMe}) can be written as $\tilde{u}_n(t) = \lambda u_{n+1}(t)$ for $n \ge 1$. It is straightforward to check that $\tilde{u}_n(t)$ satisfy the same regularized system (\ref{eq2_E1}) and (\ref{eq1_1fn}) for the cutoff parameter $N$ with the initial and boundary conditions (\ref{eq2_FMf}). Hence, $\tilde u(t) = \Phi_t^{(N,J)}(\tilde a,\tilde b)$. This proves the canonical property and the relation (\ref{eq2_E2}).

\subsection{RG eigenvalue $\rho = -1/2$}
\label{subsec_A3}

By definition of the eigenmode (\ref{eqSRG_1Y}), the eigenvalue $\rho$ is universal with respect to a choice of  canonical regularization, initial and boundary conditions. Hence, one can compute $\rho$ using relation (\ref{eqRG_R3}) for a specific regularization, initial and boundary conditions. Let us consider the regularized model (\ref{eq2_E1}) and (\ref{eq1_1fn}) with $\lambda = 2$, arbitrary $N$ and $J = 1$, and the constant boundary condition 
	\begin{equation}
	b_0(t) \equiv 1. 
	\label{eqRG_R3A1}
	\end{equation}
Initial conditions $a$ will not be important, and we denote $u^{(N)}(t) = \Phi_t^{(N,1)}(a,b)$.
	
We assume (as strongly suggested by numerically simulations) that the regularized IBVP has a stationary attractor, $u^{(N)}(t) \to \hat{u}^{(N)}$ as $t \to \infty$. This attractor depends on $N$ as specified in the superscript. We further  assume that the eigenmode approximation (\ref{eqRG_R3}) is valid in the limit $t \to \infty$, i.e., it can be formulated for the attractors as
	\begin{equation}
	\hat{u}^{(N)} \approx \hat{u}^{\infty} + \rho^N v 
	\ \ \textrm{as} \ \ N \to \infty,
	\label{eqRG_R3A2}
	\end{equation}
where $\hat{u}^{\infty}$ is a stationary state of the ideal model and $v$ is a constant sequence. 

One can verify that the stationary solution of the ideal model (\ref{eq1_1a}) and (\ref{eq1_1fn}) with boundary condition (\ref{eqRG_R3A1}) has the form 
	\begin{equation}
	\hat{u}^\infty = (2^{-n/3})_{n \ge1}.
	\label{eqRG_R3A3}
	\end{equation}
The stationary solution $\hat{u}^{(N)}$ is found as
	\begin{equation}
	\hat{u}^{(N)} = (\hat{u}^{(N)}_n)_{n \ge 1}, \quad
	\hat{u}^{(N)}_n = \left\{ \begin{array}{ll}
		2^{-n/3-3^{-2}\rho^N[(-1)^{n} 2^n-1]}, & 1 \le n \le N+1; \\[2pt]
		0, & n > N+1; 
	\end{array} \right.
	\label{eqRG_R3A4}
	\end{equation}
with $\rho = -1/2$; one can check this expression by the direct substitution into Eqs.~(\ref{eq2_E1}) and (\ref{eq1_1fn}). The factor $\rho^N$ tends to zero in the limit of vanishing regularization $N \to \infty$. The Taylor expansion of Eq.~(\ref{eqRG_R3A4}) with respect to $\rho^N$ yields Eq.~(\ref{eqRG_R3A2}) with 
	\begin{equation}
	v = (v_n)_{n \ge 1}, \quad
	v_n = \left\{ \begin{array}{ll}
	-3^{-2}\big((-1)^{n} 2^n-1\big)2^{-n/3}\log 2, & 1 \le n \le N+1; \\[2pt]
		0, & n > N+1.
	\end{array} \right.
	\label{eqRG_R3A5}
	\end{equation}
This derivation yields the eigenvalue $\rho = -1/2$, in full agreement with the numerical simulations reported in Section~\ref{sec_fixedpoint}.

\vspace{2mm}\noindent\textbf{Acknowledgments.} 
This work was supported by CNPq grant 308721/2021-7, FAPERJ grant E-26/201.054/2022 and CAPES MATH-AmSud project CHA2MAN. 

\vspace{2mm}\noindent\textbf{Data Accessibility.} 
Data available on reasonable request.

\bibliographystyle{plain}
\bibliography{refs}

\end{document}